\documentclass[a4paper,12pt]{amsart}
\title{Decompositions, approximate structure, transference, and the 
Hahn-Banach theorem}
\author{W.T. Gowers}
\address{Department of Pure Mathematics and Mathematical Statistics, 
Wilberforce Road, Cambridge CB3 0WB, UK.}
\email{W.T.Gowers@dpmms.cam.ac.uk}

\usepackage{amsmath, wrapfig}
\usepackage{dsfont, a4wide, amsthm, amssymb, amsfonts, graphicx}
\usepackage{fancyhdr, xspace, psfrag, setspace, supertabular, color}

\makeatletter
\renewcommand\subsection{\@startsection
   {subsection}{2}{0mm}
   {-\baselineskip}
   {0.5\baselineskip}
   {\normalfont\normalsize\bf}}
\makeatother

\newtheorem{theorem}{Theorem}[section]
\newtheorem{proposition}[theorem]{Proposition}
\newtheorem{lemma}[theorem]{Lemma}

\newtheorem{corollary}[theorem]{Corollary}

\newtheorem*{definition}{Definition}
\newtheorem*{decomposition*}{Hoped-for decomposition}
\newtheorem*{theorem*}{Putative Inverse Theorem}
\newtheorem*{assumptions*}{Consequences of failure of decomposition}

\def\E{\mathbb{E}}
\def\Z{\mathbb{Z}}
\def\R{\mathbb{R}}

\def\C{\mathbb{C}}
\def\N{\mathbb{N}}

\def\a{\alpha}
\def\b{\beta}
\def\g{\gamma}
\def\d{\delta}
\def\e{\epsilon}

\def\cD{\mathcal{D}}

\def\ra{\rightarrow}
\def\seq#1#2{#1_1,\dots,#1_#2}

\def\sp#1{\langle #1\rangle}
\def\ol{\overline}
\def\hf{\hat{f}}

\def\cD{\mathcal{D}}
\def\hG{\hat{G}}
\def\hg{\hat{g}}
\def\hh{\hat{h}}

\begin{document}

\maketitle

\begin{abstract} We discuss three major classes of theorems in
additive and extremal combinatorics: decomposition theorems,
approximate structure theorems, and transference principles.
We also show how the finite-dimensional Hahn-Banach theorem
can be used to give short and transparent proofs of many results
of these kinds. Amongst the applications of this method is a 
much shorter proof of one of the major steps in the proof of
Green and Tao that the primes contain arbitrarily long arithmetic 
progressions. In order to explain the role of this step, we include 
a brief description of the rest of their argument. A similar proof has
been discovered independently by Reingold, Trevisan, Tulsiani
and Vadhan \cite{rttv}.
\end{abstract}

\tableofcontents

\section{Introduction}

This paper has several purposes. One is to provide a survey of some of
the major recent developments in the rapidly growing field that has
come to be known as additive combinatorics, focusing on three
classes of theorems: decomposition theorems, approximate structural
theorems and transference principles. (An explanation of these phrases
will be given in just a moment.) A second is to show how the
Hahn-Banach theorem leads to a simple and flexible method for proving
results of these three kinds. A third is to demonstrate this by actually
giving simpler proofs of several important results, or parts of results.
One of the proofs we shall simplify is the proof of Green and Tao that
the primes contain arbitrarily long arithmetic progressions \cite{greentao}, 
which leads to the fourth purpose of this paper: to provide a partial guide to their 
paper. We shall give a simple proof of a result that is implicit in their
paper, and made explicit in a later paper of Tao and Ziegler, and then
we shall explain informally how they use this result to prove their
famous theorem. A proof along similar lines has been discovered independently
by Reingold, Trevisan, Tulsiani and Vadhan \cite{rttv}. We have tried to 
design this paper so that the reader
who is just interested in the Green-Tao theorem can get away with 
reading only a small part of it. However, the earlier sections of 
the paper provide considerable motivation for the later arguments,
so such a reader would be well-advised at least to skim the sections
that are not strictly speaking necessary.

Now let us describe the classes of theorems that will principally concern 
us. By a \textit{decomposition theorem} we mean a statement that tells
us that a function $f$ with certain properties can be decomposed as a 
sum $\sum_{i=1}^kg_i$, where the functions $g_i$ have certain other
properties. There are two kinds of decomposition theorem that have
been particularly useful. One kind says that $f$ can be written as
$\sum_{i=1}^kg_i+h$, where the functions $g_i$ have some explicit
description and $h$, the ``error term'' is in a useful sense small.

Another kind, which is closely related, brings us to our second class
of results. An \textit{approximate structure theorem} is a result that
says that, under appropriate conditions, we can write a function $f$
as $f_1+f_2$, where $f_1$ is ``structured'' in some sense, and $f_2$
is ``quasirandom''. The rough idea is that the structure of $f_1$ is
strong enough for us to be able to analyse it reasonably explicitly,
and the quasirandomness of $f_2$ is strong enough for many properties
of $f_1$ to be unaffected if we ``perturb'' it to $f_1+f_2$. Often,
in order to obtain stronger statements about the structure and the
quasirandomness, one allows also a small $L_2$-error: that is, one
writes $f$ as $f_1+f_2+f_3$ with $f_1$ structured, $f_2$ quasirandom,
and $f_3$ small in $L_2$.

A \textit{transference principle} is a statement to the effect that
a function $f$ that belongs to some space $X$ of functions can be
approximated by a function $g$ that belongs to another space $Y$.
Such a statement is useful if the functions in $Y$ are easier to
handle than the functions in $X$ and the approximation is of a 
kind that preserves the properties that one is interested in. As
we shall see later, a transference principle is the fundamental
step in the proof of Green and Tao. For now, let us merely note
that a transference principle is a particular kind of decomposition
theorem: it tells us that $f$ can be written as $g+h$, where $g\in Y$
and $h$ is small in an appropriate sense.

As well as the Green-Tao theorem, we shall discuss several other
results in additive combinatorics. One is a structure theorem proved
by Tao in an important paper \cite{tao} that gives a discretization of
Furstenberg's ergodic-theory proof \cite{furstenberg} of Szemer\'edi's
theorem \cite{szemeredi}, or more precisely a somewhat different
ergodic-theory argument due to Host and Kra \cite{hostkra}. We shall
give an alternative proof of (a slight generalization of) this theorem,
and give some idea of how it can be used to prove other results.
Amongst these other results are Roth's theorem \cite{roth}, which states 
that every set of integers of positive upper density contains an arithmetic
progression of length 3, and Szemer\'edi's regularity lemma
\cite{szreg}, a cornerstone of extremal graph theory, which shows that
every graph can be approximated by a disjoint union of boundedly many
quasirandom graphs (and which is a very good example of an approximate
structure theorem).

The remaining sections of this paper are organized as follows.
The next section introduces several norms that are used to 
define quasirandomness. Strictly speaking, it is independent
of much of the rest of the paper, since many of our results
will be rather general ones about norms that satisfy various
hypotheses. However, for the reader who is unfamiliar with the
basic concepts of additive combinatorics it may not be obvious
that these hypotheses are satisfied except in one or two very
special cases: section 2 should convince such a reader that
the general results can be applied in many interesting contexts.

In section 3, we introduce our main tool, the finite-dimensional
Hahn-Banach theorem, and we give one or two very easy consequences of
it. Even these consequences are of interest, as we shall explain---one
of them is a non-trivial decomposition theorem of the first kind
discussed above---but the method comes into its own when we introduce
one or two further ideas in order to obtain conclusions that can be
applied much more widely.

One of these ideas is the relatively standard one of
\textit{polynomial approximations}. Often we start with a function $f$
that takes values in an interval $[a,b]$, and we want its structured
part $f_1$ to take values in $[a,b]$ as well. If $f_1$ is bounded, and
if the class of structured functions is closed under composition with
polynomials, then we can sometimes achieve this by choosing a
polynomial $P$ such that $P(x)$ approximates $a$ when $x<a$, $x$ when
$a\leq x\leq b$, and $b$ when $x>b$. Then the function $Pf_1$ takes
values in $[a,b]$ (approximately), and under appropriate circumstances
it is possible to argue that it approximates $f_1$.  In section 4, we
shall illustrate this technique by proving two results. The first is a fairly
simple transference principle that we shall need later, and the second is a
slightly more complicated version of it that is needed for proving the
Green-Tao theorem. The latter is essentially the same as the
``abstract structure theorem'' of Tao and Ziegler \cite{taoziegler}, so called because
it is an abstraction of arguments from the paper of Green and Tao. It
is this second result that can be regarded as a major step in the
proof of the Green-Tao theorem, and which is used to prove their
transference principle. We shall end Section 4 with a brief description of 
the rest of the proof of Green and Tao.

In section 5, we shall prove the structure theorem of Tao mentioned 
earlier, and show how it leads to a strengthened decomposition 
theorem. We end the section, and the paper, with an indication 
of how to use the structure theorem.

\section{Some basic concepts of additive combinatorics.}

\subsection{Preliminaries: Fourier transforms and $L_p$-norms.}
 
Let $G$ be a finite Abelian group. A \textit{character} on $G$ is a
non-zero function $\psi:G\ra\C$ with the property that
$\psi(xy)=\psi(x)\psi(y)$ for every $x$ and $y$. It is easy to show
that $\psi$ must take values in the unit circle. It is also easy to
show that two distinct characters are orthogonal. To see this, note
first that if $\psi_1$ and $\psi_2$ are distinct, then
$\psi_1(\psi_2)^{-1}$ is a non-trivial character (that is, a character
that is not identically 1). Next, note that if $\psi$ is a non-trivial
character and $\psi(y)\ne 1$, then
$\E_x\psi(x)=\E_x\psi(xy)=\psi(y)\E_x\psi(x)$, so $\E_x\psi(x)=0$.
(The notation ``$\E_x$'' is shorthand for ``$|G|^{-1}\sum_{x\in G}$''.)
This implies the orthogonality. Less obvious, but a straightforward
consequence of the classification of finite Abelian groups, is the
fact that the characters span all functions from $G$ to $\C$: that is,
they form an orthonormal basis of $L_2(G)$. (We shall discuss this
space more in a moment.)

If $f:G\ra\C$, then the \textit{Fourier transform} $\hf$ of $f$ tells
us how to expand $f$ in terms of the basis of characters. More
precisely, one first defines the \textit{dual group} $\hG$ to
be the group of all characters on $G$ under pointwise multiplication.
Then $\hf$ is a function from $\hG$ to $\C$, defined by the formula
\begin{equation*}
\hf(\psi)=\E_xf(x)\ol{\psi(x)}=\E_xf(x)\psi(-x).
\end{equation*}
The \textit{Fourier inversion formula} (which it is an easy
exercise to verify) then tells us that
\begin{equation*}
f(x)=\sum_\psi\hf(\psi)\psi(x),
\end{equation*}
which gives the expansion of $f$ as a linear combination of characters.

There are two natural measures that
one can put on $G$: the uniform probability measure, and the counting
measure (which assigns measure 1 to each singleton). Both of these are
useful. The former is useful when one is looking at functions that are
``flat'': an example would be the characteristic function of a dense
subset $A\subset G$. If we write $A(x)$ for $\chi_A(x)$, then
$\E_xA(x)=|G|^{-1}\sum_xA(x)=|A|/|G|$ is the density of $A$. The
counting measure is more useful for functions $F$ that are of
``essentially bounded support'', in the sense that there is a set $K$
of bounded size such that $F$ is approximately equal (in some
appropriate sense) to its restriction to $K$. 

If $f$ is a flat function, then there is a useful sense in which its
Fourier transform is of essentially bounded support in the dual
group. Therefore, if we are interested in flat functions defined on
$G$, then we look at the uniform probability measure on $G$ and the
counting measure on the dual group $\hG$. We then define inner
products, $L_p$-norms, and $\ell_p$-norms as follows.

The inner product of two functions $f$ and $g$ from $G$ to $\C$ is
the quantity $\sp{f,g}=\E_xf(x)\ol{g(x)}$. The resulting Euclidean
norm is $\|f\|_2=\Bigl(\E_x|f(x)|^2\Bigr)^{1/2}$, and the Euclidean
space is $L_2$. More generally, $L_p$ is the space of all functions
from $G$ to $\C$, with the norm $\|f\|_p=\Bigl(\E_x|f(x)|^p\Bigr)^{1/p}$,
where this is interpreted as $\max|f(x)|$ when $p=\infty$.

On the dual group $\hG$ we have the same definitions, but with
expectations replaced by sums. Thus,
$\sp{F_1,F_2}=\sum_xF_1(x)\ol{F_2(x)}$ and
$\|F\|_p=\Bigl(\sum_x|F(x)|^p\Bigr)^{1/p}$. The resulting space is
denoted $\ell_p$. Once again, $\|F\|_\infty$ is $\max|F(x)|$, so
$L_\infty$ and $\ell_\infty$ are in fact the same space.

Two fundamental identities that are used repeatedly in additive
combinatorics are the \textit{convolution identity} and 
\textit{Parseval's identity}. The \textit{convolution} $f*g$ of two
functions $f,g:G\ra\C$ is defined by the formula
\begin{equation*}
f*g(x)=\E_{y+z=x}f(y)g(z),
\end{equation*}
and the convolution identity states that 
$(f*g)^{\wedge}(\psi)=\hf(\psi)\hg(\psi)$ for every $\psi\in\hG$.
That is, the Fourier transform ``converts convolution into
pointwise multiplication''. It also does the reverse: the 
Fourier transform of the pointwise product $fg$ is the 
convolution $\hf*\hg$, where the latter is defined by the
formula
\begin{equation*}
\hf*\hg(\psi)=\sum_{\rho\sigma=\psi}\hf(\rho)\hg(\sigma).
\end{equation*}

Parseval's identity is the simple
statement that $\sp{f,g}=\sp{\hf,\hg}$. It is important to keep
in mind that the two inner products are defined differently, one
with expectations and the other with sums, just as the convolutions
were defined differently in $G$ and $\hG$. Setting $f=g$ in Parseval's
identity, we deduce that $\|f\|_2=\|\hf\|_2$.

The group that will interest us most is the cyclic group $\Z_N=\Z/N\Z$.
If we set $\omega=\exp(2\pi i/N)$, then any function of the form
$x\mapsto\omega^{rx}$ is a character, and the functions $\omega^{rx}$
and $\omega^{sx}$ are distinct if and only if $r$ and $s$ are not
congruent mod $N$. Therefore, one can identify $\Z_N$ with its
dual, writing
\begin{equation*}
\hf(r)=\E_xf(x)\omega^{-rx}
\end{equation*}
whenever $r$ is an element of $\Z_N$. However, the measure we use 
on $\Z_N$ is different when we are thinking of it as a dual group.

The reason that Fourier transforms are important in additive
combinatorics is that many quantities that arise naturally can be
expressed in terms of convolutions, which can then be simplified by
the Fourier transform. For instance, as we shall see in the next
subsection, the quantity $\E_{x,d}f(x)f(x+d)f(x+2d)$ arises naturally
when one looks at arithmetic progressions of length 3. This can be
rewritten as $\E_{x,z}f(x)f(z)f((x+z)/2)=\E_{x,z}f(x)f(z)\ol{g(x+z)}$,
where $g(u)=\ol{f(u/2)}$. (We need $N$ to be odd for this to make
sense.)  This is the inner product of $f*f$ with $g$, so it is equal
to $\sp{\hf^2,\hg}$.

\subsection{What is additive combinatorics about?}

The central objects of study in additive combinatorics are finite
subsets of Abelian groups. For example, one of the main results in
the area, Szemer\'edi's theorem, can be formulated as follows. 

\begin{theorem} \label{szem} For every $\d>0$ and every positive 
integer $k$ there exists $N$ such that every subset $A\subset\Z_N$ of
cardinality at least $\d N$ contains an arithmetic progression of
length $k$.
\end{theorem}

\noindent Here, $\Z_N$ stands for the cyclic group $\Z/N\Z$ of
integers mod $N$, and an arithmetic progression of length $k$
means a set of the form $\{x,x+d,\dots,x+(k-1)d\}$ with $d\ne 0$.

What are the \textit{maps} of interest between finite subsets of
Abelian groups? An initial guess might be that they were restrictions
of group homomorphisms, but that turns out to be far too narrow a
definition. Instead, they are functions called Freiman homomorphisms.
A \textit{Freiman homomorphism of order} $k$ between sets $A$ and $B$
is a function $\phi:A\ra B$ such that 
\begin{equation*}
\phi(a_1)+\phi(a_2)+\dots+\phi(a_k)
=\phi(a_{k+1})+\phi(a_{k+2})+\dots+\phi(a_{2k})
\end{equation*}
whenever
\begin{equation*}
a_1+a_2+\dots+a_k=a_{k+1}+a_{k+2}+\dots+a_{2k}.
\end{equation*}
In particular, a Freiman homomorphism of order 2, often just known
as a Freiman homomorphism, is a function such that 
$\phi(a_1)+\phi(a_2)=\phi(a_3)+\phi(a_4)$ whenever $a_1+a_2=a_3+a_4$.
This is equivalent to the same definition with minus instead of
plus, which is often more convenient.

A \textit{Freiman isomorphism of order} $k$ is a Freiman homomorphism
of order $k$ with an inverse that is also a Freiman homomorphism of
order $k$. The rough idea is that a Freiman homomorphism of order $k$
preserves all the linear structure of a set $A$ that can be detected
by integer combinations with coefficients adding up to 0 and with
absolute values adding up to at most $2k$. For example, it is an easy
exercise to show that if $A$ is an arithmetic progression and $B$ is
Freiman-isomorphic to $A$, then $B$ is also an arithmetic progression.
This is because a sequence $(x_1,x_2,\dots,x_m)$ is an arithmetic
progression, written out in a sensible order, if and only if
$x_{i+2}-x_{i+1}=x_{i+1}-x_i$ for every $i$. It is also easy to
show that if $A$ and $B$ are isomorphic, then their sumsets
$A+A$ and $B+B$ have the same size. 

Thus, a more precise description of the main objects studied by 
additive combinatorics would be that they are finite subsets of
Abelian groups, up to Freiman isomorphisms of various orders.

An important aspect of results such as Szemer\'edi's theorem is that
they have a certain ``robustness''. For instance, combining
Szemer\'edi's theorem with a simple averaging argument, one can deduce
the following corollary (which was first noted by 
Varnavides~\cite{varnavides}).

\begin{corollary} \label{varn} For every $\d>0$ and every positive 
integer $k$ there exists $\e>0$ such that, for every sufficiently 
large positive integer $N$, every subset $A\subset\Z_N$ of cardinality
at least $\d N$ contains at least $\e N^2$ arithmetic progressions of
length $k$.
\end{corollary}

A second important aspect is that they have ``functional versions''.
One can regard a subset of $\Z_N$ as a function that takes values
0 and 1. It turns out that many of the arguments used to prove
Szemer\'edi's theorem apply to a much wider class of functions.
In particular, they apply to functions that take values in the
interval $[0,1]$. The following generalization of Szemer\'edi's
theorem is easily seen to follow from Corollary \ref{varn}.

\begin{corollary} \label{functszem} For every $\d>0$ and every positive 
integer $k$ there exists $\e>0$ such that, for every positive integer 
$N$ and every function $f:\Z_N\rightarrow[0,1]$
for which $\E_xf(x)\geq\d$, we have the inequality
$$\E_{x,d}f(x)f(x+d)\dots f(x+(k-1)d)\geq\e.$$
\end{corollary}

\noindent Here, $\E_{x,d}$ denotes the expectation over all pairs
$(x,d)\in\Z_N^2$. It can be regarded as shorthand for $N^{-2}\sum_{x,d}$,
but it is better to think in probabilistic terms: the left-hand side
of the above inequality is then an expectation over all arithmetic
progressions of length $k$ (including degenerate ones with $d=0$,
but for large $N$ these make a tiny contribution to the total).

A third important aspect is a deeper form of robustness. It turns 
out that quantities such as $\E_{x,d}f(x)f(x+d)\dots f(x+(k-1)d)$
are left almost unchanged if you perturb $f$ by adding a function
$g$ that is small in an appropriate norm. Furthermore, it is possible
for $g$ to be small in this norm even when the average size $\E_x|g(x)|$
of $g(x)$ is large: a typical example of such a function is one that
takes the values $\pm 1$ independently at random. The changes to
the values of $f$ are then quite large, but the randomness of $g$
forces their contribution to expressions such as 
$\E_{x,d}f(x)f(x+d)\dots f(x+(k-1)d)$ to cancel out almost 
completely. This cancellation, rather than smallness of a more 
obvious kind, is what justifies our thinking of $f+g$ as a
``perturbation'' of $f$.

Thus, it is tempting to revise further our rough definition of
additive combinatorics and say that the central objects of study
are subsets of Abelian groups, up to Freiman isomorphism and
``quasirandom perturbation''. However, it takes some effort to make
this idea precise, since the notion of a Freiman homomorphism does not
apply as well to functions as it does to sets (because it is
insufficiently robust). Also, not every quantity of importance in the
area is approximately invariant up to quasirandom perturbations: an
example of one that isn't is the size of the sumset $A+A$ of a set $A$
of size $n$. So we shall content ourselves with the observation that
all the results of this paper \textit{are} approximately invariant.

So that we can say what this means, let us give some examples of
norms that measure quasirandomness.

\subsection{Uniformity norms for subsets of finite Abelian groups.}
\label{uniformitynorms}

Let $G$ be a finite Abelian group, and let $g:G\ra\C$. The 
$U^2$-\textit{norm} of $g$ is defined by the formula 
\begin{equation*}
\|g\|_{U^2}^4=\E_{x,a,b}g(x)\ol{g(x+a)}\ol{g(x+b)}g(x+a+b).
\end{equation*}
We shall not give here the verification that this is a norm
(though it will follow from a remark we make in \ref{u2struct}),
since our main concern is the sense in which it measures
quasirandomness. It can be shown that if $f:G\ra\C$ is a 
function with $\|f\|_\infty\leq 1$ and $g$ is another such
function with the additional property that $\|g\|_{U^2}$ is
small, then
\begin{equation*}
\E_{x,d}f(x)f(x+d)f(x+2d)\approx\E_{x,d}(f+g)(x)(f+g)(x+d)(f+g)(x+2d).
\end{equation*}
A case of particular interest is when $f$ is the characteristic
function of a subset $A\subset G$ of density $\d$, which again we
shall denote by $A$, and $g(x)=A(x)-\d$ for every $x$. If
$\|g\|_{U^2}$ is small, then we can think of $f$ as a quasirandom
perturbation of the constant function $\d$. Then
$\E_{x,d}A(x)A(x+d)A(x+2d)$ will be around $\d^3$, the approximate
value it would take (with high probability) if the elements of $A$
were chosen independently at random with probability $\d$. When
$\|g\|_{U^2}$ is small, we say that $A$ is a \textit{quasirandom}
subset of $G$. (This definition is essentially due to Chung and Graham
\cite{chunggraham}.)

In many respects, a quasirandom set behaves as one would expect
a random set to behave, but in by no means all. For example, even
if $A$ is as quasirandom as it is possible for a set to be, it
does not follow that 
\begin{equation*}
\E_{x,d}A(x)A(x+d)A(x+2d)A(x+3d)\approx\d^4.
\end{equation*}
An example that shows this is the subset $A\subset\Z_N$ that
consists of all $x$ such that $x^2\in[-\d N/2,\d N/2]$. It can
be shown that the density of $A$ is very close to $\d$ when
$N$ is large, and that $\|g\|_{U^2}=\|A-\d\|_{U^2}$ is extremely
small. However, for this set $A$, 
\begin{equation*}
\E_{x,d}A(x)A(x+d)A(x+2d)A(x+3d)
\end{equation*}
turns out to be at least $c\d^3$ for some absolute constant $c>0$. 
We will not prove this here (a proof can be found in \cite{quadex}), but 
we give the example in order to draw attention to its quadratic
nature. It turns out that this feature of the example is necessary,
though quite what that means is not obvious, and the proof is
even less so. See subsection \ref{inverse} for further discussion
of this.

This example shows that the smallness of the $U^2$-norm is not
sufficient to explain all the typical behaviour of a random function.
For this one needs to introduce ``higher'' uniformity norms, of 
which the next one is (unsurprisingly) the $U^3$-norm. If $g$ is
a function, then $\|g\|_{U^3}^8$ is given by the expression
\begin{equation*}
\E_{x,a,b,c}g(x)\ol{g(x+a)g(x+b)g(x+c)}g(x+a+b)
g(x+a+c)g(x+b+c)\ol{g(x+a+b+c)}.
\end{equation*}
From this it is easy to guess the definition of the $U^k$ norm,
but for completeness here is a formula for it:
\begin{equation*}
\|g\|_{U^k}^{2^k}=\E_{x,a_1,\dots,a_k}\prod_{\e\in\{0,1\}^k}
C^{|\e|}g\bigl(x+\sum\e_ia_i\bigr),
\end{equation*}
where $C$ denotes the operation of complex conjugation and
$|\e|$ denotes the number of non-zero coordinates of $\e$.

These norms were introduced in \cite{gowers}, where it was shown, as 
part of a proof of Szemer\'edi's theorem, that if $A$ is a 
subset of $\Z_N$ of density $\d$, $g(x)=A(x)-\d$ for every
$x$, and $\|g\|_{U^k}$ is sufficiently small (meaning smaller
than a positive constant that depends on $\d$ but not on $N$), 
then
\begin{equation*}
\E_{x,d}A(x)A(x+d)\dots A(x+kd)\approx\d^{k+1}.
\end{equation*}
Let us call a set \textit{uniform of degree} $k-1$ if its $U^k$-norm
is small. Then the above assertion is that a set of density $\d$ that
is sufficiently uniform of degree $k-1$ contains roughly as many
arithmetic progressions (mod $N$) of length $k+1$ as a random set 
of density $\d$ will (with high probability) contain. In particular,
if $A$ is \textit{quadratically uniform} (meaning that the $U^3$-norm
of $A-\d$ is sufficiently small), then 
\begin{equation*}
\E_{x,d}A(x)A(x+d)A(x+2d)A(x+3d)\approx\d^4.
\end{equation*}

The arithmetic progression $\{x,x+d,\dots,x+(k-1)d\}$ can be thought 
of as a collection of $k$ linear forms in $x$ and $d$. It can be
shown that for any collection of linear forms in any number of
variables, there exists a $k$ such that every set $A$ that is 
sufficiently uniform of degree $k$ contains about as many 
of the corresponding linear configurations as a random set
of the same density. This was shown by Green and Tao \cite{linprimes}, 
who generalized the argument in \cite{gowers}. The question of precisely which
$U^k$ norm is needed is a surprisingly subtle one. It is conjectured
in \cite{gowerswolf} that the answer is the smallest $k$ for which the $k$th
powers of the linear forms in question are linearly independent.
For instance, the configuration $\{x,x+d,x+2d,x+3d\}$ needs the
$U^3$-norm because $x-2(x+d)+(x+2d)=x^2-3(x+d)^2+3(x+2d)^2-(x+3d)^2=0$, 
but the cubes are linearly independent. A special case of this
result is proved in \cite{gowerswolf} using ``quadratic Fourier analysis'',
which we will discuss in \ref{higherfanal}: to prove the full
conjecture would require a theory of higher-degree Fourier
analysis that will probably exist in due course but which has 
not yet been sufficiently developed.

\subsection{Uniformity norms for graphs and hypergraphs.}

There are very close and important parallels between uniformity
of subsets of finite Abelian groups, and quasirandomness of graphs
and hypergraphs. For this reason, even though the relevant parts
of graph and hypergraph theory belong to extremal combinatorics, 
they have become part of additive combinatorics as well: one
could call them additive combinatorics without the addition.

Since that may seem a peculiar thing to say, let us briefly see
what these parallels are. Let $G$ be a graph on $n$ vertices.
One can think of $G$ as a two-variable function $G(x,y)$,
where $x$ and $y$ are vertices and $G(x,y)=1$ if $xy$ is an
edge and $0$ otherwise. Just as we may regard a subset $A$ of a
finite Abelian group as quasirandom if a certain norm of $A-\d$
is small, we can regard a graph as quasirandom if a certain
norm of the function $G-\d$ (where now $\d$ is the density
$\E_{x,y}G(x,y)$ of the graph $G$) is small. This norm is
given by the formula
\begin{equation*}
\|g\|_{GU^2}^4=\E_{x,x',y,y'}g(x,y)\ol{g(x,y')g(x',y)}g(x',y'),
\end{equation*}
which makes sense, and is useful, whenever $X$ and $Y$ are finite sets
and $g:X\times Y\ra\C$.  The theory of quasirandom graphs was
initiated by Thomason \cite{thomason} and more fully developed by Chung, Graham
and Wilson \cite{cgw}. The definition we have just given is equivalent to
the definition in the latter paper.
    
To see how this relates to the $U^2$ norm, let $X$ and $Y$ equal
a finite Abelian group $\Gamma$, let $f:\Gamma\ra\C$ and let 
$g(x,y)=f(x+y)$. Then
\begin{equation*}
\|g\|_{GU^2}^4=\E_{x,x',y,y'}f(x+y)\ol{f(x+y')f(x'+y)}f(x'+y').
\end{equation*}
The quadruples $(x+y,x+y',x'+y,x'+y')$ are uniformly distributed
over all quadruples $(a,b,c,d)$ such that $a+d=b+c$. Since the
same is true of all quadruples of the form $(x,x+a,x+b,x+a+b)$,
we see that the right-hand side of the above formula is nothing
other than $\|f\|_{U^2}$. 

A similar argument can be used to relate the higher-degree uniformity
norms to notions of quasirandomness for $k$-uniform \textit{hypergraphs}, 
which are like graphs except that instead of having edges, which
are pairs of vertices, one has hyperedges, which are $k$-tuples of
vertices. The following formula defines a norm on $k$-variable
functions:
\begin{equation*}
\|g\|_{HU^k}^{2^k}=\E_{x_1^0,x_1^1}\dots\E_{x_k^0,x_k^1}
\prod_{\e\in\{0,1\}^k}C^{|\e|}f(x_1^{\e_1},\dots,x_k^{\e_k}).
\end{equation*}
If $f(x_1,\dots,x_k)$ has the form $g(x_1+\dots+x_k)$, then
$\|f\|_{HU^k}=\|g\|_{U^k}$. 

A hypergraph $H$ of density $\d$ behaves in many
respects like a random hypergraph of density $\d$ if 
$\|H-\d\|_{HU^k}$ is small enough. For instance, if $k=3$, 
then the \textit{simplex density}, which is given by
the expression
\begin{equation*}
\E_{x,y,z,w}H(x,y,z)H(x,y,w)H(x,z,w)H(y,z,w)
\end{equation*}
is roughly $\d^4$, or what it would be in the random case.
More generally, if $\|H_i-\d\|_{HU^k}$ is small for 
$i=1,2,3,4$, then 
\begin{equation*}
\E_{x,y,z,w}H_1(x,y,z)H_2(x,y,w)H_3(x,z,w)H_4(y,z,w)
\end{equation*}
is again roughly $\d^4$. This assertion, which is proved
by repeated use of the Cauchy-Schwarz inequality (see
\cite{hypergraphs} for a more general result), implies that 
\begin{equation*}
\E_{x,y,z,w}A(-3x-2y-z)A(-2x-y+w)A(-x+z+2w)A(y+2z+3w)\approx\d^4
\end{equation*}
whenever $A$ is a subset of $\Z_N$ such that $\|A-\d\|_{U^3}$ is
small. But the four linear forms above form an arithmetic progression
of length 4 and common difference $x+y+z+w$: this is a sketch of what
turns out to be the most natural proof that the $U^3$-norm controls
arithmetic progressions of length 4.

These ideas can be developed to give a complete proof of Szemer\'edi's
theorem: see \cite{naglerodlschacht}, \cite{rodlskokan}, \cite{hypergraphs}, 
\cite{taohyper}.

\subsection{Easy structure theorems for the $U^2$-norm.}\label{u2struct}

A great deal of information about the $U^2$-norm comes from the 
following simple observation.

\begin{lemma}\label{u2fourier}
Let $G$ be a finite Abelian group and let $f:G\ra\C$. Then
$\|f\|_{U^2}=\|\hf\|_4$.
\end{lemma}

\begin{proof} By the convolution identity and Parseval's identity,
\begin{equation*}
\|f\|_{U^2}^4=\E_{x+y=z+w}f(x)f(y)\ol{f(z)f(w)}=\sp{f*f,f*f}
=\sp{\hf^2,\hf^2}=\sum_\psi|\hf(\psi)|^4.
\end{equation*}
The result follows on taking fourth roots. 
\end{proof}

Now let us suppose that $\|f\|_2\leq 1$, and let us fix 
some small constant $\eta>0$. Then the number of characters
$\psi$ such that $|\hf(\psi)|\geq\eta$ is at most $\eta^{-2}$,
since $\sum_\psi|\hf(\psi)|^2=\|\hf\|_2^2=\|f\|_2^2=1$. Using
the inversion formula and this fact, we can split $f$ into two 
parts, $\sum_{\psi\in K}\hf(\psi)\psi$ and 
$\sum_{\psi\notin K}\hf(\psi)\psi$, where $K$ is the set of
all $\psi$ such that $|\hf(\psi)|\geq\eta$. Let us call these
two parts $g$ and $h$, respectively. The function $g$ 
involves a bounded number of characters, and characters are
functions that we can describe completely explicitly. Therefore,
it can be thought of as the ``structured'' part of $f$. As for
$h$, it is quasirandom, since
\begin{equation*}
\|h\|_{U^2}^4=\|\hh\|_4^4\leq\|\hh\|_2^2\|\hh\|_\infty^2
\leq\eta^2\|\hf\|_2^2\leq\eta^2.
\end{equation*}

Unfortunately, this simple decomposition turns out not to be 
very useful, for reasons that we shall explain later. We mention 
it in order to put some of our later results in perspective. The
same applies to the next result, which shows that one can obtain
a much stronger relationship between $\|h\|_{U^2}$ and the upper 
bound on the size of $K$ if one is prepared to tolerate a small
$L_2$-error as well. This result and its proof are part of the
standard folklore of additive combinatorics.

\begin{proposition}\label{easystrong}
Let $f$ be a function from a finite Abelian group $G$ to $\C$ and
suppose that $\|f\|_2\leq 1$. Let $\eta:\R_+\ra\R_+$ be a positive
decreasing function that tends to 0 and let $\e>0$. Then there is a
positive integer $m$ such that $f$ can be written as $f_1+f_2+f_3$,
where $f_1$ is a linear combination of at most $m$ characters,
$\|f_2\|_{U^2}\leq\eta(m)$, and $\|f_3\|_2\leq\e$.
\end{proposition}

\begin{proof}
Let $N=|G|$ and let us enumerate the dual group $\hG$ as
$\psi_1,\dots,\psi_N$ in such a way that the absolute values of
the Fourier coefficients $\hf(\psi_i)$ are in non-increasing order.
Choose an increasing sequence of positive integers $m_1,m_2,\dots$
in such a way that $m_{r+1}\geq\eta(m_r)^{-4}$ for every $r$.

Now let us choose $i$ and attempt to prove the result using the
decomposition $f_1=\sum_{i\leq m_r}\hf(\psi_i)\psi_i$,
$f_2=\sum_{i>m_{r+1}}\hf(\psi_i)\psi_i$, and
$f_3=\sum_{m_r<i\leq m_{r+1}}\hf(\psi_i)\psi_i$. Then $f_1$
is a linear combination of at most $m_r$ characters. Since
there can be at most $m_{r+1}$ characters $\psi$ with
$|\hf(\psi)|\geq m_{r+1}^{-1/2}$, we find that 
\begin{equation*}
\|f_2\|_{U^2}^4=\|\hf_2\|_4^4\leq m_{r+1}^{-1}\|\hf_2\|_2^2
\leq\eta(m_r)^4\|\hf\|_2^2\leq\eta(m_r)^4.
\end{equation*}
Therefore, we are done if $\|f_3\|_2\leq\e$. But the possible
functions $f_3$ (as $r$ varies) are disjoint parts of the Fourier 
expansion of $f$, so at most $\e^{-2}$ of them can have norm
greater than $\e$. Therefore, we can find $r\leq\e^{-2}$ such
that the proposed decomposition works. 
\end{proof}

There is nothing to stop us taking $m_1=1$. The proof then gives us
the desired decomposition for some $m$ that is bounded above by a
number that results from starting with 1 and applying the function
$t\mapsto\eta(t)^{-4}$ at most $\e^{-2}$ times.  Although Proposition
\ref{easystrong} is still not all that useful, it resembles other
results that are, as we shall see in due course. In those results, it
is common to require $\eta(m)$ to be exponentially small: the
resulting bound is then of tower type.

\subsection{Inverse theorems.}\label{inverse}

A \textit{direct theorem} in additive number theory is one that
starts with a description of a set and uses that description to
prove that the set has certain additive properties. For instance,
the statement that every positive integer is the sum of four
squares starts with the explicitly presented set $S$ of all perfect 
squares, and proves that the four-fold sumset $S+S+S+S$ is the
whole of $\N\cup\{0\}$. An \textit{inverse theorem} is a result
that goes in the other direction: one starts with a set $A$ that
is assumed to have certain properties, and attempts to find 
some kind of description of $A$ that explains those properties. 
Ideally, this description should be so precise that it actually
characterizes the properties in question: a set $A$ has the
properties if and only if it satisfies the description.

A remarkable inverse theorem, which lies at the heart of many recent
results in additive combinatorics, is a theorem of Freiman 
\cite{freiman} (later given a considerably more transparent proof 
by Ruzsa \cite{ruzsa}) that characterizes sets
that have small sumsets. If $A$ is a set of $n$ integers, then it is
easy to show that the sumset $A+A$ has size at least $2n-1$ and at
most $n(n+1)/2$.  What can be said about $A$ if the size of the sumset
is close to its minimum, in the sense that $|A+A|\leq C|A|$ for some
fixed constant $C$? A simple example of such a set is an arithmetic
progression. A slightly less simple example is a set $A$ that 
is contained in an arithmetic progression of length at most 
$Cn/2$. A less simple example altogether is a ``two-dimensional
arithmetic progression'': that is, a set of the form
$\{x_0+rd_1+sd_2:0\leq r<t_1, 0\leq s<t_2\}$. If $A$ is such
a set, then $|A+A|\leq 4|A|$, and more generally if $A$ is a
$k$-dimensional arithmetic progression (the definition of
which is easy to guess), then $|A+A|\leq 2^k|A|$. As in the
one-dimensional case, one can pass to large subsets and obtain
more examples. Freiman's theorem states that one has then 
exhausted all examples.

\begin{theorem}\label{freiman}
For every $C$ there exist $k$ and $K$ such that every set
$A$ of $n$ integers such that the sumset $A+A$ has size
at most $Cn$ is contained in an arithmetic progression
of dimension at most $k$ and cardinality at most $Kn$. 
\end{theorem}

Freiman's theorem has been extremely influential, in large part
because of Ruzsa's proof, which was extremely elegant and
conceptual, and gave much better bounds than Freiman's argument. These
bounds have subsequently been improved by Chang \cite{chang}, who added further
interesting ingredients to Ruzsa's argument. A generalization of Freiman's
theorem to subsets of an arbitrary Abelian group was proved by Green
and Ruzsa \cite{greenruzsa}.

The notion of an inverse theorem makes sense also for functions
defined on Abelian groups. For instance, here is a simple inverse
theorem about functions with large $U^2$-norm.

\begin{proposition} \label{u2inverse}
Let $c>0$, let $G$ be a finite Abelian group,
let $f:G\ra\C$ be a function such that $\|f\|_2\leq 1$ and suppose
that $\|f\|_{U^2}\geq c$.  Then there exists a character $\psi$ such
that $|\sp{f,\psi}|\geq c^2$.
\end{proposition}

\begin{proof} By Lemma \ref{u2fourier} and our assumptions about $f$, 
\begin{equation*}
c^2\leq\|\hf\|_4^2\leq\|\hf\|_\infty\|\hf\|_2\leq\|\hf\|_\infty,
\end{equation*}
which is what is claimed.
\end{proof}

Conversely, and without any assumption about $\|f\|_2$, if there
exists a character $\psi$ such that $\sp{f,\psi}\geq c$, then
$\|f\|_{U^2}=\|f\|_4\geq c^{1/4}$. Therefore, correlation with a
character ``explains'' the largeness of the $U^2$-norm.

What about the $U^3$-norm? This turns out to be a much deeper question.
As our remarks earlier have suggested, quadratic functions come into
play when one starts to think about it. For example, if $f:\Z_N\ra\C$
is the function $x\mapsto\omega^{2rx^2}$ for some $r$ (where $\omega$
is once again equal to $\exp(2\pi i/N)$), then the identity
\begin{equation*}
x^2-(x+a)^2-(x+b)^2-(x+c)^2+(x+a+b)^2+(x+a+c)^2+(x+b+c)^2
-(x+a+b+c)^2=0
\end{equation*}
implies easily that $\|f\|_{U^3}=1$. However, it is also easy
to show that $f$ does not correlate significantly with any 
character. Therefore, we are forced to consider quadratic 
functions. If $q$ is a quadratic function, then let us call 
the function $\omega^q$ a \textit{quadratic phase function}.

It is tempting to conjecture that a bounded function $f$ with 
$U^3$-norm at least $c$ must correlate with a quadratic phase function,
meaning that $\E_xf(x)\omega^{q(x)}\geq c'$ for some quadratic
function $q$ and some constant $c'$ that depends on $c$ only. However,
although such a correlation is a sufficient condition for the
$U^3$-norm of $f$ ot be large, it is not necessary, because there are 
``multidimensional'' examples. For instance, if $P$ is the
two-dimensional arithmetic progression
$\{x_0+rd_1+sd_2:0\leq r<t_1, 0\leq s<t_2\}$, then we can define
something like a quadratic form $q$ on $P$ by the formula
$q(x_0+rd_1+sd_2)=ar^2+brs+cs^2$. We can then define a function
$f$ to be $\omega^{q(x)}$ when $x\in P$ and $0$ otherwise. Let us
call such a function a \textit{generalized quadratic phase function}.
It is not hard to prove that such functions have large $U^3$ norms, and
that they do not have to correlate with ordinary quadratic phase
functions. 

In \cite{gowers}, the following ``weak inverse theorem'' was proved for all 
$U^k$ norms. 

\begin{theorem}
Let $c>0$ be a constant and let $f:\Z_N\ra\C$ be a function such that
$\|f\|_\infty\leq 1$ and $\|f\|_{U^k}\geq c$. Then there is a
partition of $\Z_N$ into arithmetic progressions $P_i$ of length at
least $N^{\a(c,k)}$, and for each $P_i$ there is a polynomial $r_i$ of
degree at most $k$ such that, writing $\pi_i$ for the density
$|P_i|/N$ of $P_i$, we have $\sum_i\pi_i|\E_{x\in
P_i}f(x)\omega^{r_i(x)}|\geq c/2$.
\end{theorem}

This result was the main step in the proof of Szemer\'edi's theorem
given in \cite{gowers}. The reason that this is a ``weak inverse theorem'' is 
that the converse is far from true. The result shows that $f$ correlates
with a function that is made out of many fragments of polynomial
phase functions, but it does not provide what one might hope for:
correlation with a single generalized polynomial phase function.
However, the proof strongly suggested that such a result should be
true, and Green and Tao, by adding some important further ingredients, have 
established a strong inverse theorem in the quadratic case \cite{greentaou3}. 
Let us state their result a little imprecisely.

\begin{theorem}\label{u3inverse}
Let $c>0$ be a constant and let $f:\Z_N\ra\C$ be a function such that
$\|f\|_\infty\leq 1$ and $\|f\|_{U^k}\geq c$. Then there exists a
constant $c'$ that depends on $c$ only, and a generalized quadratic
phase function $g$, such that $|\sp{f,g}|\geq c'$.
\end{theorem}

\subsection{Higher Fourier analysis.}\label{higherfanal}

As we have seen, the $U^2$-norm of a function $f$ is equal to the 
$\ell_4$-norm of its Fourier transform, and this observation leads 
quickly to a decomposition of functions into a structured part and
a quasirandom part. Is there a comparable result for the $U^3$-norm? 
The inverse theorem of Green and Tao suggests that we should try
to decompose $f$ into generalized quadratic phase functions. However,
there are far more than $N$ of these, so they do not form an 
orthonormal basis, or indeed a basis of any kind. One might nevertheless
hope for some canonical way of decomposing a function, but it is
far from clear that there is one---certainly, nobody has come close
to finding one. 

However, one can still hope for a structure theorem that resembles
Proposition \ref{easystrong}. We would expect it to say that a 
function $f$ can be decomposed into a linear combination of a
small number of generalized quadratic phase functions, plus a
function with very small $U^3$-norm, plus a function that is
small in $L_2$. Green and Tao deduced such a result from their
inverse theorem, and thereby initiated a form of \textit{quadratic
Fourier analysis}. In \cite{gowerswolf}, a different method was given for 
deducing somewhat different decomposition theorems from inverse 
theorems. The main ingredient of this method was the Hahn-Banach
theorem: the proof will be sketched in the next section. This
gave an alternative form of quadratic Fourier analysis, which
provided much better bounds for the results of that paper (the ones
that concerned controlling systems of linear forms with $U^k$-norms).

We shall have more to say about higher Fourier analysis later in
the paper.

\subsection{Easy structure theorems for graphs.}

We have already seen that the $U^2$-norm of the one-variable function
$g$, defined on a finite Abelian group $G$, can be regarded as the
$GU^2$-norm of the two-variable function $f(x,y)=g(x+y)$. The
relationship does not stop here, however. If $\psi$ is a character,
then, for any $x$, 
\begin{equation*}
\E_yg(x+y)\psi(-y)=\psi(x)\E_yg(x+y)\psi(-x-y)=\hg(\psi)\psi(x)
\end{equation*}
This shows that characters are similar to eigenvectors of the
symmetric matrix $f(x,y)$, except that they are mapped to 
multiples of their complex conjugates. However, it is notable
that the corresponding ``eigenvalues'' are the Fourier coefficients
of the function $g$. This observation suggests, correctly as it
turns out, that eigenvalues play a similar role for real symmetric 
matrices to the role played by Fourier coefficients for functions
defined on finite Abelian groups.

We briefly illustrate this by proving a result that is analogous to
Proposition \ref{easystrong}. First, we prove a well-known lemma
relating the $GU^2$-norm to eigenvalues. It will tie in better with
our previous notation (and with applications of matrices to graph
theory) if we use a slightly unconventional association between
matrices and linear maps, as we did above. Given a matrix $f(x,y)$ and
a function $u(y)$ we shall think of $fu(x)$ as the quantity
$\E_yf(x,y)u(y)$ rather than the same thing with a sum.

The finite-dimensional spectral theorem tells us that a real
symmetric matrix $f(x,y)$ has an orthonormal basis of eigenvectors. If
these are $u_1,\dots,u_n$ and the corresponding eigenvalues are
$\lambda_1,\dots,\lambda_n$, then we can express this by saying that
\begin{equation*}
f(x,y)=\sum_i\lambda_iu_i\otimes u_i,
\end{equation*}
where $u\otimes v$ denotes the function $u(x)v(y)$. To see why these
are the same, consider the effect of each side in turn on a basis 
vector $u_j$. On the one hand, we have $\E_yf(x,y)u_j(y)=\lambda_iu_j(x)$
(by our unconventional definition of matrix multiplication) while on
the other we have
\begin{equation*}
\E_y\sum_i\lambda_iu_i\otimes u_i(x,y)u_j(y)
=\sum_i\lambda_iu_i(x)\E_yu_i(y)u_j(y)=\sum_i\lambda_iu_i(x)\d_{ij}
=\lambda_ju_j(x)
\end{equation*}
by the orthonormality (with respect to the $L_2$-norm) of the 
eigenvectors.

\begin{lemma}
Let $X$ be a finite set and let $f$ be a symmetric real-valued 
function defined on $X^2$. Let the eigenvalues of $f$ be
$\lambda_1,\dots,\lambda_n$. Then $\|f\|_{GU^2}^4=\sum_r\lambda_r^4$.
\end{lemma}

\begin{proof}
All results of this kind are proved by expanding the expression for
$\|f\|_{GU^2}^4$ in terms of the spectral decomposition 
$\sum_r\lambda_ru_r\otimes u_r$ of $f$. 
\begin{eqnarray*}
\|f\|_{GU^2}^4&=&\E_{x,x'}\E_{y,y'}f(x,y)f(x,y')f(x',y)f(x',y')\\
&=&\E_{x,x'}\E_{y,y'}\sum_{p,q,r,s}\lambda_p\lambda_q\lambda_r\lambda_s
u_p(x)u_p(y)u_q(x)u_q(y')u_r(x')u_r(y)u_s(x')u_s(y')\\
&=&\sum_{p,q,r,s}\lambda_p\lambda_q\lambda_r\lambda_s
\d_{pq}\d_{pr}\d_{rs}\d_{qs}\\
&=&\sum_r\lambda_r^4\\
\end{eqnarray*}
as claimed.
\end{proof}

A similar but easier proof establishes that $\|f\|_2^2=\sum_r\lambda_r^2$.

The next result is a direct analogue for symmetric two-variable 
real functions (and therefore in particular for graphs) of
Proposition \ref{easystrong}

\begin{proposition}\label{easygraphstrong}
Let $X$ be a finite set and let $f$ be a symmetric real-valued
function on $X^2$ such that $\|f\|_2\leq 1$. Let $\eta:\R_+\ra\R_+$ be
a positive decreasing function that tends to 0 and let $\e>0$. Then
there is a positive integer $m$ such that $f$ can be written as
$f_1+f_2+f_3$, where $f_1$ is a linear combination of at most $m$
orthonormal functions of the form $u\otimes u$, 
$\|f_2\|_{GU^2}\leq\eta(m)$, and $\|f_3\|_2\leq\e$.
\end{proposition}

\begin{proof}
Let $N=|X|$ and let us enumerate an orthonormal basis $(u_i)_1^N$ 
of eigenvectors of $f$ in such a way that the absolute values of
the eigenvalues $\lambda_i$ are in non-increasing order.
Choose an increasing sequence of positive integers $m_1,m_2,\dots$
in such a way that $m_{r+1}\geq\eta(m_r)^{-4}$ for every $r$.

Now let us choose $i$ and attempt to prove the result using the
decomposition $f_1=\sum_{i\leq m_r}\lambda_iu_i\otimes u_i$,
$f_2=\sum_{i>m_{r+1}}\lambda_iu_i\otimes u_i$, and
$f_3=\sum_{m_r<i\leq m_{r+1}}\lambda_iu_i\otimes u_i$. Then $f_1$
is a linear combination of at most $m_r$ eigenvectors, which are
orthonormal to each other by the spectral theorem. By the remark
above about the sum of the squares of the eigenvalues, 
there can be at most $m_{r+1}$ eigenvectors $u_i$ with
$|\lambda_i|\geq m_{r+1}^{-1/2}$. Therefore,  
\begin{equation*}
\|f_2\|_{U^2}^4=\sum_{i>m_{r+1}}\lambda_i^4\leq m_{r+1}^{-1}\sum_i\lambda_i^2
=m_{r+1}^{-1}\|f\|_2^2\leq\eta(m_r)^4.
\end{equation*}
Therefore, we are done if $\|f_3\|_2\leq\e$. But the possible
functions $f_3$ (as $r$ varies) are disjoint parts of the spectral
expansion of $f$, so at most $\e^{-2}$ of them can have norm
greater than $\e$. Therefore, we can find $r\leq\e^{-2}$ such
that the proposed decomposition works. 
\end{proof}

The above result is closely related to a ``weak regularity lemma''
due to Frieze and Kannan \cite{friezekannan}.

\section{The Hahn-Banach theorem and simple applications.}

Let us begin by stating the version of the Hahn-Banach theorem that we
shall need. 

\begin{theorem}\label{hbbasic}
Let $K$ be a convex body in $\R^n$ and let $f$ be an element of
$\R^n$ that is not contained in $K$. Then there is a constant
$\b$ and a non-zero linear functional $\phi$ such that $\sp{f,\phi}\geq\b$
and $\sp{g,\phi}\leq\b$ for every $g\in K$. 
\end{theorem}

Now let us prove two corollaries, both of which are useful for proving
decomposition theorems.

\begin{corollary}\label{hb1}
Let $\seq K r$ be closed convex subsets of $\R^n$, each containing 0, let
$\seq c r$ be positive real numbers and suppose that $f$ is an element
of $\R^n$ that cannot be written as a sum $f_1+\dots+f_r$ with $f_i\in
c_iK_i$. Then there is a linear functional $\phi$ such that
$\sp{f,\phi}>1$ and $\sp{g,\phi}\leq c_i^{-1}$ for every $i\leq r$
and every $g\in K_i$.
\end{corollary}

\begin{proof}
Let $K$ be the convex body $\sum_ic_iK_i$. Our hypothesis is that
$f\notin K$. Since $K$ is closed, it follows that there exists $\e>0$
such that $(1+\e)^{-1}f\notin K$. Therefore, by Theorem \ref{hbbasic},
there is a constant $\b$ and a linear functional $\phi$ such that
$(1+\e)^{-1}\sp{f,\phi}\geq \b$ and $\sp{g,\phi}\leq \b$ for every
$g\in K$. Again using the fact that $K$ is closed, we can add a small
Euclidean ball $B$ to $K$ in such a way that $(1+\e)^{-1}f\notin
B+K$. Since $0\in K$, it follows that $\b>0$. Therefore, we can divide
$\phi$ by $\b$ and get $\b$ to be 1, with the result that
$\sp{f,\phi}\geq(1+\e)\b$. Since $0$ belongs to each $K_i$, we can
also conclude that $\sp{g,\phi}\leq 1$ for every $g\in c_iK_i$, which
completes the proof.
\end{proof}

\begin{corollary}\label{hb2}
Let $\seq K r$ be closed convex subsets of $\R^n$, each containing 0
and suppose that $f$ is an element of $\R^n$ that cannot be written as
a convex combination $c_1f_1+\dots+c_rf_r$ with $f_i\in K_i$. Then
there is a linear functional $\phi$ such that $\sp{f,\phi}>1$ and
$\sp{g,\phi}\leq 1$ for every $i\leq r$ and every $g\in K_i$.
\end{corollary}

\begin{proof}
Let $K$ be the set of all convex combinations $c_1f_1+\dots+c_rf_r$ 
with $f_i\in K_i$. Then $K$ is a closed convex set and $f$ is not
contained in $K$. Therefore, there exists $\e>0$ such that
$(1+\e)^{-1}f\notin K$. By Theorem \ref{hbbasic} there is a 
functional $\phi$ and a constant $\b$ such that 
$(1+\e)^{-1}\sp{f,\phi}\geq\b$ and $\sp{g,\phi}\leq\b$ whenever
$g$ belongs to $K$. In particular, $\sp{g,\phi}\leq\b$ whenever
$g$ belongs to one of the sets $K_i$. As in the proof of the
previous corollary, $\b$ must be positive and can therefore
be assumed to be 1. The result follows.
\end{proof}

Recall that if $\|.\|$ is a norm on $\R^n$, then the dual
norm $\|.\|^*$ is defined by the formula
$\|\phi\|^*=\max\{\sp{f,\phi}:\|f\|\leq 1\}$. If $f\in\R^n$ then
Theorem \ref{hbbasic} implies that there exists a functional $\phi$
such that $\|\phi\|^*\leq 1$ and $\sp{f,\phi}=\|f\|$.  Such a
functional is called a \textit{support functional} for $f$. In this
paper it will be convenient to call $\phi$ a support functional if
$\phi\ne 0$ and $\sp{f,\phi}=\|f\|\|\phi\|^*$, so that a positive
scalar multiple of a support functional is also a support functional.

The following lemma is useful in proofs that involve the Hahn-Banach
theorem, as we shall see in section \ref{decompositions}. It tells us
that the dual of an $\ell_1$-like combination of norms is an
$\ell_\infty$-like combination of their duals. We shall adopt the
convention that if $\|.\|$ is a norm defined on a subspace $V$ of
$\R^n$ then its dual $\|.\|^*$ is the seminorm defined by the
formula $\|f\|^*=\max\{\sp{f,g}:g\in V, \|g\|\leq 1\}$.

\begin{lemma}\label{maxdual}
Let $\Sigma$ be a set and for each $\sigma\in\Sigma$ let
$\|.\|_\sigma$ be a norm defined on a subspace $V_\sigma$ of
$\R^n$. Suppose that $\sum_{\sigma\in\Sigma}V_\sigma=\R^n$, and define
a norm $\|.\|$ on $\R^n$ by the formula
\begin{equation*}
\|x\|=\inf\{\|x_1\|_{\sigma_1}+\dots+\|x_k\|_{\sigma_k}:x_1+\dots+x_k=x,
\sigma_1,\dots,\sigma_k\in\Sigma\}
\end{equation*}
Then this formula does indeed define a norm, and its
dual norm $\|.\|^*$ is given by the formula
\begin{equation*}
\|z\|^*=\max\{\|z\|_{\sigma}^*:\sigma\in\Sigma\}
\end{equation*}
\end{lemma}

\begin{proof}
It is a simple exercise to check that the expression does
indeed define a norm. 

Let us begin by supposing that $\|z\|_{\sigma}^*\geq 1$ for some
$\sigma\in\Sigma$.  Then there exists $x\in V_\sigma$ such that
$\|x\|_{\sigma}\leq 1$ and $|\langle x,z\rangle|\geq 1$. But then
$\|x\|\leq 1$ as well, from which it follows that $\|z\|^*\geq
1$. Therefore, $\|z\|^*$ is at least the maximum of the
$\|z\|_{\sigma}^*$.

Now let us suppose that $\|z\|^*>1$.
This means that there exists $x$ such that $\|x\|\leq 1$
and $|\langle x,z\rangle|\geq 1+\epsilon$ for some
$\epsilon>0$. Let us choose $x_1,\dots,x_k$ such that
$x_i\in V_{\sigma_i}$ for each $i$, $x_1+\dots+x_k=x$, and 
$\|x_1\|_{\sigma_1}+\dots+\|x_k\|_{\sigma_k}<1+\epsilon$.
Then 
\begin{equation*}
\sum_i|\langle x_i,z\rangle|>\|x_1\|_{\sigma_1}+\dots+\|x_k\|_{\sigma_k}
\end{equation*}
so there must exist $i$ such that 
$|\langle x_i,z\rangle|>\|x_i\|_{\sigma_i}$, from which it follows
that $\|z\|_i^*>1$. This proves that $\|z\|^*$ is at most
the maximum of the $\|z\|_i^*$. 
\end{proof}

A particular case that will interest us is when $\Sigma$ is a
subset of $\R^n$, for each $\sigma\in\Sigma$ the subspace $V_\sigma$
is just the subspace generated by $\sigma$, and the norm on 
$V_\sigma$ is $\|\lambda\sigma\|_\sigma=|\lambda|$. The dual
seminorm is then $\|f\|_\sigma^*=|\sp{f,\sigma}|$. Thus, if we
specialize Lemma \ref{maxdual} to this case then we obtain the
following corollary. 

\begin{corollary}\label{convexhulldual}
Let $\Sigma\subset\R^n$ be a set that spans $\R^n$ and define a 
norm $\|.\|$ on $\R^n$ by the formula
\begin{equation*}
\|f\|=\inf\bigl\{\sum_{i=1}^k|\lambda_i|:f=\sum_{i=1}^k\lambda_i\sigma_i,
\ \sigma_1,\dots,\sigma_k\in\Sigma\}.
\end{equation*}
Then this formula does indeed define a norm, and its dual norm
$\|.\|^*$ is defined by the formula 
$\|f\|^*=max\{|\sp{f,\sigma}|:\sigma\in\Sigma\}$.
\end{corollary}

\subsection{A simple structure theorem.}

We now prove a very simple (and known) decomposition result that 
illustrates our basic method.

\begin{proposition}\label{dualdecomposition}
Let $\|.\|$ be any norm on $\R^n$ and let $f$ be any function in
$\R^n$. Then $f$ can be written as $g+h$ in such a way that
$\|g\|+\|h\|^*\leq\|f\|_2$.
\end{proposition}

\begin{proof}
Suppose that the result is false. We shall apply Corollary \ref{hb2}
to the function $f/\|f\|_2$, with $K_1$ and $K_2$ taken to be the unit
balls of $\|.\|$ and $\|.\|^*$. Our hypothesis is equivalent to the
assertion that $f/\|f\|_2$ is not a convex combination $c_1g_1+c_2g_2$
with $g_i\in K_i$ for $i=1,2$. Therefore, we obtain a functional
$\phi$ such that $\sp{f,\phi}>\|f\|_2$ and $\|\phi\|^*$ and $\|\phi\|$
are both at most 1. But the first property implies, by the
Cauchy-Schwarz inequality, that $\|\phi\|_2>1$, while the second
implies that $\|\phi\|_2^2=\sp{\phi,\phi} \leq\|\phi\|\|\phi\|^*\leq
1$. This is a contradiction.
\end{proof}

A simple modification of Proposition \ref{dualdecomposition} makes it
a little more flexible. Suppose that we wish to write $f$ as $g+h$ 
with $\|g\|$ small and $\|h\|^*$ not
too large. If we define a new norm $|.|$ to be $\e^{-1}\|.\|$,
then $|.|^*=\e\|.\|^*$. Applying Proposition \ref{dualdecomposition}
to these rescaled norms, we find that we can write $f$ as $g+h$ in
such a way that $\e^{-1}\|g\|+\e\|h\|^*\leq\|f\|_2$. In particular, if
$\|f\|_2=1$, then $\|g\|\leq\e$ and $\|h\|^*\leq\e^{-1}$.

The reason such a result might be expected to be useful in additive
combinatorics is that, as demonstrated in the previous section, we 
have a good supply of norms $\|.\|$ that measure quasirandomness.
Moreover, their duals, as we shall see later, can be thought of as
a sort of measure of structure. Perhaps the simplest example that
illustrates this is if we look at functions $f$ defined on
finite Abelian groups, and take $\|f\|$ to be $\|\hf\|_\infty$. 
If $\|f\|_2\leq 1$, then 
\begin{equation*}
\|f\|_{U^2}^2=\|\hf\|_4^2\leq\|\hf\|_2\|\hf\|_\infty\leq\|\hf\|_\infty,
\end{equation*}
a calculation we have already done. This shows that for functions with
bounded $L_2$-norm there is a rough equivalence between $\|f\|_{U^2}$
and $\|\hf\|_\infty$, in the sense that if one is small then so is the
other.

Thus, if $\|f\|_2\leq 1$ then $\|f\|=\|\hf\|_\infty$ being small tells
us that $f$ is quasirandom. The dual norm, $\|f\|^*=\|\hf\|_1$, is
a sort of measure of structure, since if $\|\hf\|_1$ is at most $C$, then
$f$ is a small multiple of a convex combination of trigonometric
functions, which can be approximated in $L_2$ by a linear
combination of a bounded number of such functions. Thus, we recover
a result that resembles Proposition \ref{easystrong}. It is weaker,
however, because we have not yet related the quasirandomness constant
to the structure constant by means of an arbitrary function.
However, this is easily done, again with the help of an $L_2$
error term, as the next result shows.

\begin{proposition}\label{strongdecomposition}
Let $f$ be a function in $\R^n$ with $\|f\|_2\leq 1$ and let
$\|.\|$ be any norm on $\R^n$. Let $\e>0$ and let 
$\eta:\R_+\rightarrow\R_+$ be any decreasing positive function.
Let $r=\lceil 2\e^{-1}\rceil$ and define a sequence $C_1,\dots,C_r$
by setting $C_1=1$ and $C_i=2\eta(C_{i-1})^{-1}$ when $i>1$.
Then there exists $i\leq r$ such that $f$ can be decomposed as
$f_1+f_2+f_3$ with 
$$C_i^{-1}\|f_1\|^*+\eta(C_i)^{-1}\|f_2\|+\e^{-1}\|f_3\|_2\leq 1.$$ 
In particular, $\|f_1\|^*\leq C_i$, $\|f_2\|\leq\eta(C_i)$ and 
$\|f_3\|_2\leq\e$.
\end{proposition}

\begin{proof}
If there is no such decomposition for $i$, then by Corollary \ref{hb2}
there is a functional $\phi_i$ such that $\|\phi_i\|\leq C_i^{-1}$,
$\|\phi_i\|^*\leq\eta(C_i)^{-1}$, $\|\phi_i\|_2\leq\e^{-1}$, and
$\sp{\phi_i,f}>1$. If this is true for every $i\leq r$ then 
$$\|\phi_1+\dots+\phi_r\|_2\geq\sp{\phi_1+\dots+\phi_r,f}\geq r,$$
where the first inequality follows from Cauchy-Schwarz and the
assumption that $\|f\|_2\leq 1$.

On the other hand, if $i<j$ then 
$$\sp{\phi_i,\phi_j}\leq\|\phi_i\|\|\phi_j\|^*\leq \eta(C_i)^{-1}C_j^{-1}
\leq 1/2,$$
the last inequality following from the way we constructed the sequence 
$C_1,\dots,C_r$. Therefore, 
$$\|\phi_1+\dots+\phi_r\|_2^2\leq\e^{-1}r+r(r-1)/2.$$
This contradicts the previous estimate, since $r\geq 2\e^{-1}$.
\end{proof}

It is easy to deduce Proposition \ref{easystrong} from this result.
Of course, this fact on its own is not a very convincing demonstration
of the utility of the Hahn-Banach theorem, since for the norm
$\|f\|=\|\hf\|_\infty$ it is easy to write down an explicit
decomposition of $f$, as we saw in section \ref{u2struct}. But there
are other important norms where this is certainly not the case. For
example, if we take $\|f\|$ to be $\|f\|_{U^k}$ for a larger $k$,
then there is no obvious decomposition of $f$ from which we can
read off three functions $f_1$, $f_2$ and $f_3$ with the required
properties.

\subsection{Deducing decomposition theorems from inverse theorems.}
\label{decompositions}

The main result of \cite{gowerswolf} shows that certain linear
configurations occur with the ``expected'' frequency in any set $A$
for which the balanced function $A-\d$ (where $\d$ is the density of
$A$) has sufficiently small $U^2$-norm. The interest in the result is
that the $U^2$-norm suffices for the configurations in question,
whereas the natural arguments that generalize the proof that the $U^k$
norm controls progressions of length $k-2$ would suggest that the
$U^3$-norm was needed. In order to prove the result, a form of
quadratic Fourier analysis was needed, as we have already
mentioned. The approach in \cite{gowerswolf} was to apply directly a
result of Green and Tao, which obtains a decomposition of a bounded
function $f$ by constructing an averaging projection $P$ with the
property that $f-Pf$ has small $U^3$ norm. However, there was a
technicality involved that forced us to use an iterated version of
their result that gives rise to very weak bounds.  In order to obtain
reasonable bounds for the problem, it turned out to be
convenient---indeed, as far as we could tell, necessary---to prove a
decomposition theorem that could be regarded as a quadratic analogue
of Proposition \ref{easystrong}, with the important difference that
the strong dependence of $\eta(m)$ on $m$ was not needed. (This was
why it was possible to obtain good bounds.)

The argument appears in \cite{gowerswolf2}, and it can be regarded as a special
case of a general principle that can be informally summarized as
follows: \emph{to each inverse theorem there is a corresponding
decomposition theorem}. It is possible to give a formal statement, as
will be clear from our discussion, but in practice it is much easier
to describe a \textit{method} for deducing decompositions from inverse
theorems than it is to state an artificial lemma that declares that
the method works. The main reason for this is that when one applies
the method, one typically starts with the decomposition one wants to
prove and the inverse theorem one \textit{can} prove, and adjusts the
former until it follows from the latter. We shall reflect this in our
discussion below: more precisely, we shall assume that a decomposition
of a certain general kind does \textit{not} exist, draw an easy 
consequence from this, and see when this consequence contradicts
any given inverse theorem.

Suppose, then, that we have a subset $\Sigma\subset\R^n$ of functions
that we regard as ``structured'', and suppose that the functions in
$\Sigma$ span $\R^n$.  Suppose also that we have another function $f$
that we would ideally like to decompose as a linear combination
$\sum_{i=1}^k\lambda_i\sigma_i$ of functions $\sigma_i\in\Sigma$ with
$\sum_{i=1}^k|\lambda_i|$ not too large, together with some error
terms. That is, we look for a result of the following kind.
\begin{decomposition*}
The function $f$ can be written in the form
\begin{equation*}
f=\sum_{i=1}^k\lambda_i\sigma_i+g_1+\dots+g_r,
\end{equation*}
where $\sum_{i=1}^k|\lambda_i|\leq M$, each $\sigma_i$ belongs to $\Sigma$,
and for each $j\leq r$ we have an
inequality of the form $\|g_j\|_{(j)}\leq\eta_j$. 
\end{decomposition*}
\noindent Typically, $r$ will be a very small integer such as 2.

Lemma \ref{maxdual} says that the formula
\begin{eqnarray*}
\|g\|&=&\inf\{\sum_{i=1}^k|\lambda_i|:
g=\sum_{i=1}^k\lambda_i\sigma_i,\ \sigma_i\in\Sigma\}\\
&=&\inf\{\sum_{i=1}^k\|g\|_{\sigma_i}:g=g_1+\dots+g_k,
\ \sigma_1,\dots,\sigma_k\in\Sigma,\ g_i\in V_{\sigma_i}\}\\
\end{eqnarray*}
defines a norm, and that the dual of this norm is the norm
\begin{equation*}
\|\phi\|^*=\max_{\sigma\in\Sigma}|\sp{\sigma,\phi}|.
\end{equation*}

Now let us suppose that no decomposition of the kind we are looking
for exists. This is equivalent to the assumption that $f$ has no
decomposition of the form $g_0+g_1+\dots+g_k$ with $\|g_0\|\leq M$ and
$\|f_i\|_{(i)}\leq\eta_i$ for every $i$. If this is the case, then by
Corollary \ref{hb1} there must be a linear functional $\phi$ such that
$\sp{f,\phi}>1$, $\|\phi\|^*\leq M^{-1}$, and
$\|\phi\|_{(i)}^*\leq\eta_i^{-1}$ for $i=1,2,\dots,r$. 

The statement that $\|\phi\|^*\leq M^{-1}$ tells us that 
$|\sp{\sigma,\phi}|\leq M^{-1}$ for every $\sigma\in\Sigma$.
Thus, what we would like is an inverse theorem that concludes
the opposite: that there must be some $\sigma\in\Sigma$ such
that $|\sp{\sigma,\phi}|>M^{-1}$. Before we think about this,
let us list the assumptions that we have at our disposal.

\begin{assumptions*} Suppose that there is no decomposition
$f=\sum_{i=1}^k\lambda_i\sigma_i+g_1+\dots+g_r$ such that
$\sum_{i=1}^k|\lambda_i|\leq M$, each $\sigma_i$ belongs to $\Sigma$,
and $\|g_j\|_{(j)}\leq\eta_j$ for each $j\leq r$. Then there exists
$\phi$ such that

(i) $\sp{\sigma,\phi}\leq M^{-1}$ for every $\sigma\in\Sigma$;

(ii) $\sp{f,\phi}>1$;

(iii) $\|\phi\|_{(j)}^*\leq\eta_j^{-1}$ for $j=1,2,\dots,r$.
\end{assumptions*}

The assumptions of an inverse theorem are typically that $f$
is not too big in one norm, such as, for instance, the 
$L_\infty$-norm, but not too small in another, such as the 
$U^3$-norm. The only information we have that could possibly
imply a lower bound on any norm of $\phi$ is the inequality
$\sp{f,\phi}>1$, and even that does not help unless we have
an \textit{upper} bound on some norm of $f$. (Of course, it
is hardly surprising that such a bound would be required
for a theorem that allows us to decompose $f$ into a bounded
combination of bounded functions.)

So let us suppose that we have an inverse theorem of the 
following form. (We have introduced the constant $K$ to
allow us to multiply $\phi$ by an arbitrary non-zero scalar.)

\begin{theorem*}
Let $\phi\in\R^n$ be a function such that $\|\phi\|\leq K$ and 
$|||\phi|||\geq\e$. Then there exists $\sigma\in\Sigma$ such
that $|\sp{\sigma,\phi}|\geq Kc(\e/K)$.
\end{theorem*}

\noindent This will be contradicted under the following circumstances:

(a) the upper bounds $\|\phi\|_{(i)}^*\leq\eta_i^{-1}$ imply that 
$\|\phi\|\leq K$;

(b) the upper bounds on the $\|\phi\|_{(i)}^*$, an upper bound
on some norm of $f$, and the lower bound $\sp{f,\phi}>1$, 
together imply that $|||\phi|||\geq\e$;

(c) $M^{-1}< Kc(\e/K)$.

\noindent For example, suppose that $M^{-1}<Kc(\e\eta)$ and we would
like a decomposition $f=\sum_{i=1}^k\lambda_i\sigma_i+g+h$ with
$\sum_{i=1}^k|\lambda_i|\leq M$, $|||g|||\leq\e$ and
$\|h\|^*\leq\eta$. If such a decomposition does not exist,
then we obtain $\phi$ such that $\sp{\sigma,\phi}<\eta^{-1}c(\e\eta)$
for every $\sigma\in\Sigma$, $|||\phi|||^*\leq\e^{-1}$,
$\|\phi\|\leq\eta^{-1}$, and $\sp{f,\phi}>1$. If we also know 
that $\|f\|_2\leq 1$, then it follows that $\|\phi\|_2\geq 1$.
But since $\|\phi\|_2^2\leq |||\phi|||.|||\phi|||^*$, it
follows that $|||\phi|||\geq\e$. This contradicts the
inverse theorem (with $K=\eta^{-1}$).

If we know a little bit more about $f$, then we can obtain a
correspondingly stronger result. For instance, suppose that
we know that $|||f|||^*\leq \e^{-1}$. Then the bound $\sp{f,\phi}>1$
immediately implies that $|||\phi|||>\e$, so we do not need
the error term $g$ in the decomposition.


Decomposition results obtained by the simple argument above---just
assume that a decomposition doesn't exist, apply Hahn-Banach, and
contradict an inverse theorem---can be very useful. However, in
order to use them one has to do a little more work. For example,
it is not usually trivial that a sum of the form 
$\sum_{i=1}^k\lambda_i\sigma_i$ is ``structured'', even if the
sum $\sum_{i=1}^k|\lambda_i|$ is smallish and all the individual
functions $\sigma_i$ are highly structured. The difficulty is that
$k$ may be very large, and in order to deal with it one tends to 
need a principle that says that functions $\sigma_i$ are either
``closely related'' or ``far apart''. A simple example is when
$\Sigma$ is the set of all characters, in which case any two
elements of $\Sigma$ are either identical or orthogonal. In \cite{gowerswolf2}
a lemma was proved to the effect that two generalized quadratic
phases were either ``linearly related'' or ``approximately
orthogonal''. That made it possible to replace the linear 
combination by a much smaller linear combination of slightly
more general functions.

A second point is that one sometimes wants more information about
the ``structured function'' $f_1=\sum_{i=1}^k\lambda_i\sigma_i$. For
instance, if $\|f\|_\infty\leq 1$ it can be extremely helpful to
know that $\|f_1\|_\infty\leq 1$ as well. This does not come
directly out of the method above, but it does when we combine
that method with methods that we shall discuss in the next
section.

Just before we finish this section, we observe that inverse theorems
can be used to prove strengthened decomposition theorems as well: that
is, ones where some of the $\eta_i$ can be made to depend on
$M$. Suppose, for example, that our inverse theorem tells us that
whenever $\|\phi\|_\infty\leq 1$ and $\|\phi\|\geq\e$ there must exist
$\sigma\in\Sigma$ such that $|{\sigma,\phi}|\geq c(\e)$. Suppose also
that (as often happens) $\|f\|^*\geq\|f\|_\infty$ for every
$f\in\R^n$. Now let $f$ be a function with $\|f\|_2\leq 1$ and use
Proposition \ref{easystrong} to write $f$ as $f_1+f_2+f_3$ with
$\|f_1\|^*\leq C$, $\|f_2\|\leq\eta(C)$ and $\|f_3\|_2\leq\theta$. In our
discussion just after the statement of the putative inverse theorem,
we observed that knowing that $\|f_1\|^*\leq C$ would yield a
decomposition $f_1=\sum_{i=1}^k\lambda_i\sigma_i+h$, where
$\sum_{i=1}^k|\lambda_i|\leq c(\theta C^{-2})$ (taking $\e=C^{-1}$, 
$K=C$, and replacing $\eta$ by $\theta$), and $\|h\|_1\leq\theta$.
Therefore, we can decompose $f$ as 
$\sum_{i=1}^k\lambda_i\sigma_i+f_2+f_3+h$, with 
$\sum_{i=1}^k|\lambda_i|\leq c(\theta C^{-2})$, $\|f_2\|\leq\eta(C)$,
and $\|f_3+h\|_1\leq 2\theta$. Since $\eta$ is an arbitrary function,
we can make it depend in an arbitrary way on $c(\theta C^{-2})$. Thus,
we have obtained the following result. (Note that the constants and 
functions are not the same as the constants and functions with
the same names in the discussion that has just finished.)
\begin{theorem}\label{decompositiontheorem}
Let $\Sigma$ be a subset of $\R^n$ that spans $\R^n$.
Let $\|.\|$ be a norm such that $\|f\|_\infty\leq\|f\|^*$ for every
$f\in\R^n$. Suppose that for every function $f$ with $\|f\|_\infty\leq 1$
and $\|f\|\geq\e$ there exists $\sigma\in\Sigma$ such that
$|\sp{f,\sigma}|\geq c(\e)$. Let $\theta>0$ and let $\eta$ be a 
decreasing function from $\R_+$ to $\R_+$. Then there exists a constant
$C_0$, depending on $\eta$ and $\theta$ only, such that every function 
$f\in\R^n$ with $\|f\|_2\leq 1$ has a decomposition
\begin{equation*}
f=\sum_{i=1}^k\lambda_i\sigma_i+f_2+f_3
\end{equation*}
with the following property: each $\sigma_i$ belongs to $\Sigma$ and 
there is a constant $C\leq C_0$ such that $\sum_{i=1}^k|\lambda_i|\leq C$, 
$\|f_2\|\leq\eta(C)$, and $\|f_3\|_1\leq\eta$. 
\end{theorem}

\section{The positivity and boundedness problems.}

Although our results so far are sometimes useful, they have a serious
limitation. Suppose, for example, that we wish to use Proposition
\ref{strongdecomposition}. What we would like to do is use the
structural properties of $h$ to prove that certain quantities, such as
$\E_{x,d}h(x)h(x+d)h(x+2d)$, are large, and then to show that $f=g+h$
is a ``random enough'' perturbation of $h$ for
$\E_{x,d}f(x)f(x+d)f(x+2d)$ to be large as well. But even if $h$ is a
small linear combination of just a few trigonometric functions, there
is no particular reason for $\E_{x,d}h(x)h(x+d)h(x+2d)$ to be
large. If we want it to be large, then we need additional
assumptions. The most useful one in practice is \textit{positivity}.

Suppose that $f\in\R^n$ is a function with $\|f\|_2\leq 1$ and 
that it takes non-negative values. With an appropriate choice of
norm $\|.\|$, Proposition \ref{strongdecomposition} allows us to 
decompose $f$ into a ``structured part'', a ``quasirandom part''
and a small $L_2$ error. One's intuition suggests that the
structured part of a non-negative function should not need to 
take negative values, and this turns out to be correct for the
norms discussed in section \ref{uniformitynorms}.

In section \ref{taostruct} we shall prove a very general
result of this kind. In this section, we shall prove some simpler
results that illustrate the method of polynomial approximations;
we shall use this method repeatedly later. 

\subsection{Algebra norms, polynomial approximation and a first 
transference theorem.}

To begin with, we need a definition that will pick out the class of
norms for which we can prove results. Actually, for now we shall give
a definition that is not always broad enough to be useful. In the next
section we shall define a broader class of norms to which the method
still applies.

\begin{definition} Let $X$ be a finite set. An \emph{algebra 
norm} on $\R^X$ is a norm $\|.\|$ such that $\|fg\|\leq\|f\|\|g\|$
for any two functions $f$ and $g$, and $\|\mathbf{1}\|=1$. 
\end{definition}

A good example of an algebra norm---indeed, the central example---is
the $\ell_1$-norm of the Fourier transform of $f$, which
has the submultiplicativity property because 
\begin{equation*}
\|\widehat{fg}\|_1=\|\hf*\hg\|_1\leq\|\hf\|_1\|\hg\|_1.
\end{equation*}
The predual of this norm (it is of course the dual as well but
we shall be thinking of it as the primary norm and the algebra
norm as its dual) is the $\ell_\infty$ norm of $\hf$, which,
as we have already seen, is in a crude sense equivalent to the
$U^2$-norm for many functions of interest.

We shall use the following simple lemma repeatedly.

\begin{lemma}\label{algebrabasics}
Let $\|.\|$ be a norm on $\R^n$ such that the dual norm $\|.\|^*$ is
an algebra norm. Then $\|f\|\geq|\E_xf(x)|$ and
$\|f\|^*\geq\|f\|_\infty$ for every function $f$.
\end{lemma}

\begin{proof}
Since $\|.\|^*$ is an algebra norm, $\|\mathbf{1}\|^*=1$, so
$\|f\|\geq|\sp{f,\mathbf{1}}|=|\E_xf(x)|$.

For the second part, if $\|f\|^*\leq 1$ then $\|f^n\|^*\leq 1$ for 
every $n$. It follows that $\|f\|_\infty\leq 1$, since otherwise
at least one coordinate of $f_n$ would be unbounded. Therefore,
$\|f\|_\infty\leq\|f\|^*$ for every $f$. 
\end{proof}

The Weierstrass approximation theorem tells us that every continuous
function on a closed bounded interval can be uniformly approximated
by polynomials. It will be helpful to define a function connected
with this result. Given a real polynomial $P$, let $R_P$ be the
polynomial obtained from $P$ by replacing all the coefficients
of $P$ by their absolute values. If $J:\R\ra\R$ is a continuous 
function, $C$ is a positive real number and $\d>0$, let $\rho(C,\d,J)$
be twice the infimum of $R_P(C)$ over all polynomials $P$ such that
$|P(x)-J(x)|\leq\d$ for every $x\in[-C,C]$. So that it will not be
necessary to remember the definition of $\rho(C,\d,J)$ we now state
and prove a simple but very useful lemma.

\begin{lemma}\label{polyapprox1}
Let $\|.\|^*$ be an algebra norm, let $J:\R\ra\R$ be a continuous
function and let $C$ and $\d$ be positive real numbers. Then there
exists a polynomial $P$ such that $\|P\phi-J\phi\|_\infty\leq\d$
and $\|P\phi\|^*\leq\rho(C,\d,J)$ for every $\phi\in\R^n$ such 
that $\|\phi\|^*\leq C$.
\end{lemma}

\begin{proof}
It is immediate from the definition of $\rho(C,\d,J)$ that for every
$C$ and every $\d>0$ there exists a polynomial $P$ such that
$|P(x)-J(x)|\leq\d$ for every $x\in[-C,C]$, and such that
$R_P(C)\leq\rho(C,\d,J)$.

Now let $\phi\in\R^n$ be a function with $\|\phi\|^*\leq C$. Then
$\|\phi\|_\infty\leq C$ as well, since $\|.\|^*$ is an algebra norm.
Since $P$ and $J$ agree to within $\d$ on $[-C,C]$, it follows that
$\|P\phi-J\phi\|_\infty\leq\d$. 

Suppose that $P$ is the polynomial $P(x)=a_nx^n+\dots+a_1x+a_0$. 
Then, by the triangle inequality and the algebra property of $\|.\|^*$,
\begin{eqnarray*}
\|P\phi\|&\leq&|a_n|\|\phi^n\|^*+\dots+|a_1|\|\phi\|^*+|a_0|\\
&\leq&|a_n|(\|\phi\|^*)^n+\dots+|a_1|\|\phi\|^*+|a_0|\\
&=&R_P(\|\phi\|^*).\\
\end{eqnarray*}
Since the coefficients of $R_P$ are all non-negative, this is at
most $R_P(C)$, which is at most $\rho(C,\d,J)$, by our choice of $P$.
\end{proof}

In more qualitative terms, the above lemma tells us that if $\phi$ is
bounded in an algebra norm and we compose it with an arbitrary
continuous function $J$, then the resulting function $J\phi$ can be
uniformly approximated by functions that are still bounded in the
algebra norm.



The next result is our first transference theorem of the paper. It
tells us that if $\mu$ and $\nu$ are non-negative functions on a set
$X$ and they are sufficiently close in an appropriate norm, then
any non-negative function that is dominated by $\mu$ can be 
``transferred to''---that is, approximated by---a non-negative
function that is dominated by $\nu$. We shall apply this principle
in Section 5. As we shall see later in this section, it is also
not hard to generalize the result to obtain a generalized version
of the Green-Tao transference theorem.

\begin{theorem}\label{transference}
Let $\mu$ and $\nu$ be non-negative functions on a set $X$ and suppose
that $\|\mu\|_1$ and $\|\nu\|_1$ are both at most 1. Let
$\eta,\d>0$, let $J:\R\ra\R$ be the function given
by $J(x)=(x+|x|)/2$ and let $\e=\d/2\rho(\eta^{-1},\d/4,J)$. Let $\|.\|$ be a
norm on $\R^X$ such that the dual norm $\|.\|^*$ is an algebra norm
and suppose that $\|\mu-\nu\|\leq\e$. Then for every function $f$ with
$0\leq f\leq\mu$ there exists a function $g$ such that $0\leq
g\leq\nu(1-\d)^{-1}$ and $\|f-g\|\leq\eta$.
\end{theorem}

\begin{proof} An equivalent way of stating the conclusion is that
$f=g+h$ with $0\leq g\leq\nu(1-\d)^{-1}$ and $\|h\|\leq\eta$. Thus,
if the result is false then we can find a functional
$\phi$ such that $\sp{f,\phi}>1$, but $\sp{g,\phi}\leq 1$ for
every $g$ such that $0\leq g\leq\nu(1-\d)^{-1}$, and 
$\|\phi\|^*\leq\eta^{-1}$.

The first condition on $\phi$ is equivalent to the statement that
$\sp{\nu,\phi_+}\leq 1-\d$. To see this, note that for any $\phi$,
the $g$ that maximizes $\sp{g,\phi}$ takes the value $0$ when
$\phi(x)<0$ and $\nu(x)(1-\d)^{-1}$ when $\phi(x)\geq 0$, in which
case $\sp{g,\phi}=(1-\d)^{-1}\sp{\nu,\phi_+}$. 

Now $\phi_+$ is equal to $J\phi$. Since $\|.\|^*$ is an algebra norm,
we can apply Lemma \ref{polyapprox1} and obtain a polynomial $P$ such
that $\|P\phi-\phi_+\|_\infty\leq\d/4$ and $\|P\phi\|^*\leq
R_P(C)=\rho(\eta^{-1},\d/4,J)$, which we shall abbreviate to $\rho$.

Since $\sp{\nu,\phi_+}\leq 1-\d$ and $\|\nu\|_1\leq 1$, it follows that 
$\sp{\nu,P\phi}\leq 1-3\d/4$. Since $\|P\phi\|^*\leq\rho$ and 
$\|\mu-\nu\|\leq\e$, it follows that $\sp{\mu,P\phi}\leq 1-3\d/4+\e\rho$.
Since $\|\mu\|_1\leq 1$, it follows that $\sp{\mu,\phi_+}\leq 1-\d/2+\e\rho$.
Since $f\leq\mu$ it follows that $\sp{f,\phi_+}\leq 1-\d/2+\e\rho$,
and since $f\geq 0$ it follows that $\sp{f,\phi}\leq 1-\d/2+\e\rho$,
which is a contradiction.
\end{proof}

\subsection{Approximate duality and algebra-like structures.}\label{approxdual}

As the previous section shows, we can carry out
polynomial-approximation arguments when we are looking at a norm
$\|.\|$ for which the dual norm $\|.\|^*$ is an algebra norm. A key
insight of Green and Tao (which has received less comment than other
aspects of their proof) is that one can carry out
polynomial-approximation arguments under hypotheses that are weaker in
two respects: one can use pairs of norms that are not precisely dual to
each other, and the norm that measures structure can have much weaker
properties than those of an algebra norm. It is not hard to generalize
the arguments in an appropriate way: the insight was to see that there
were important situations in which one could obtain the weaker hypotheses
even when the stronger ones were completely false.

To see why this might be, think once again about the one
algebra norm we have so far considered, namely $\|\hf\|_\infty$. For
bounded functions $f$, this is closely related (by Proposition
\ref{u2inverse} and the remark after it) to $\|\hf\|_4$, which equals
the $U^2$-norm, so we can deduce facts related to the $U^2$-norm from
the fact that $\|\hf\|_1$ is an algebra norm.

We can regard this argument as carrying out the following procedure.
First, we establish an inverse theorem for the $U^2$-norm: this is
what we did in Proposition \ref{u2inverse}. We then note that the
functions that we obtain in the inverse theorem, namely the characters,
are closed under pointwise multiplication. And then we make the following
observation.

\begin{lemma}\label{definingalgebranorms}
Let $X$ be a set of functions in $\C^n$ that spans all of $\C^n$, 
contains the constant function $\mathbf{1}$, and is closed under pointwise
multiplication. Suppose also that $\|\phi\|_\infty\leq 1$ for every
function $\phi\in X$. Then the norm $\|.\|$ on $\R^n$ defined by the 
formula
\begin{equation*}
\|f\|=\inf\bigl\{\sum_{i=1}^k|\lambda_i|:f_1,\dots,f_k\in X, 
\ f=\sum_{i=1}^k\lambda_if_i\bigr\}
\end{equation*}
is an algebra norm.
\end{lemma}

\begin{proof}
Suppose that $f=\sum_{i=1}^k\lambda_if_i$ and
$g=\sum_{j=1}^l\mu_jg_j$, with all $f_i$ and $g_j$ in $X$. Then 
$fg=\sum_{i=1}^k\sum_{j=1}^l\lambda_i\mu_jf_ig_j$. Since $X$ is
closed under pointwise multiplication, each $f_ig_j$ belongs to 
$X$. Moreover, $\sum_{i=1}^k\sum_{j=1}^l|\lambda_i||\mu_j|
=\sum_{i=1}^k|\lambda_i|\sum_{j=1}^l|\mu_j|$. From this the
submultiplicativity follows easily. The fact that $\|\mathbf{1}\|=1$
follows from the assumption that $\mathbf{1}\in X$ and that all
functions in $X$ have $L_\infty$-norm at most 1. 
\end{proof}

In the case where $X$ is the set of all characters on a finite Abelian
group, the norm given by Lemma \ref{definingalgebranorms} is the
$\ell_1$-norm of the Fourier transform.

Now suppose that we want to prove comparable facts about the $U^3$-norm.
An obvious approach would be to use Theorem \ref{u3inverse}, the
inverse theorem for the $U^3$-norm. However, the generalized quadratic
phase functions that appear in the conclusion of that theorem are
not quite closed under pointwise multiplication: associated with
them are certain parameters that one wants to be small, which obey
rules such as $\gamma(fg)\leq\gamma(f)+\gamma(g)$. 

As we shall see, this is not a serious difficulty, because often
one can restrict attention to products of a bounded number of
functions that an inverse theorem provides. A more fundamental
problem is that for the higher $U^k$-norms we do not (yet) have
an inverse theorem. Or at least, we do not have an inverse theorem
where the function that appears in the conclusion can be explicitly
described. What Green and Tao did to get round this difficulty was
to define a class of functions that they called basic anti-uniform
functions, and to prove a ``soft'' inverse theorem concerning those
functions. 

\begin{definition} 
For every function $f\in\R^n$, let $\cD f$ be
the function defined by the formula
\begin{equation*}
\cD f(x)=\E_{a,b,c}f(x+a)f(x+b)f(x+c)\ol{f(x+a+b)f(x+a+c)f(x+b+c)}
f(x+a+b+c).
\end{equation*}
Let $X$ be a subset of $\R^n$. A \emph{basic anti-uniform function} 
(with respect to $X$) is a function of the form $\cD f$ with $f\in X$.
\end{definition} 

Needless to say, the above definition generalizes straightforwardly
to a class of basic anti-uniform functions for the $U^k$-norm,
for any $k$. The same applies to the next proposition.

\begin{proposition}\label{u3softinverse}
Let $X$ be a subset of $\R^n$ and let $f\in X$ be a function such that
$\|f\|_{U^3}\geq\e$. Then there is a basic anti-uniform function $g$,
with respect to $X$, such that $\sp{f,g}\geq\e^8$.
\end{proposition}

\begin{proof}
The way we have stated the result is rather artificial, since the
basic anti-uniform function in question is nothing other than $\cD f$.
Moreover, it is trivial that $\sp{f,\cD f}\geq\e^8$, since if we expand
the left-hand side we obtain the formula for $\|f\|_{U^3}^8$.
\end{proof}

Of course, the price one pays for such a simple proof is that one has
far less information about basic anti-uniform functions than one would
have about something like a generalized polynomial phase function. In
particular, it is not obvious what one can say about products of basic
anti-uniform functions.

We remark here that an inequality proved in \cite{gowers} implies 
easily that $\sp{g,\cD f}\leq\|g\|_{U^3}\|f\|_{U^3}^7$ for every function
$g$. Thus, $\|\cD f\|_{U^3}^*\leq\|f\|_{U^3}^7$. Since $\sp{f,\cD
f}=\|f\|_{U^3}\|f\|_{U^3}^7$, we see that $\cD f$ is a support
functional for $f$. It is not hard to show that it is unique (up to a
scalar multiple). Since every function is a support functional for
something, it may seem as though there is something odd about the
definition of a basic anti-uniform function. However, it is less
all-encompassing than it seems, because we are restricting attention
to functions $\cD f$ for which $f$ belongs to some specified class of
functions $X$. (Nevertheless, we shall usually choose $X$ in such a way
that every function is a multiple of a basic anti-uniform function.)

A crucial fact that Green and Tao proved about basic anti-uniform
functions is that, for suitable sets $X$, their products have 
$(U^k)^*$-norms that can be controlled. To be precise, they proved
the following lemma. (It is not stated as a lemma, but rather as
the beginning step in the proof of their Lemma 6.3.)

\begin{lemma}\label{bauproducts} For every positive integer $K$ there 
is a constant $C_K$ such that if $\cD f_1,\dots,\cD f_K$ are basic
anti-uniform functions [with respect to a suitable set $X$], then
$\|\cD f_1\dots\cD f_K\|_{U^k}^*\leq C_K$.
\end{lemma}

\subsection{A generalization of the Green-Tao-Ziegler transference theorem.}\label{GTZ}

We shall be more concerned with the form of Lemma \ref{bauproducts}
than with the details of what $X$ is, since our aim is to describe
in an abstract way the important properties of the operator
$f\mapsto\cD f$. This we do in the next definition, which is
meant to capture the idea that the dual of a certain norm 
somewhat resembles an algebra norm.

\begin{definition}
Let $\|.\|$ be a norm on $\R^n$ such that $\|f\|_\infty\leq\|f\|^*$
for every $f\in\R^n$, and let $X$ be a bounded subset 
of $\R^n$. Then $\|.\|$ is a \emph{quasi algebra predual norm},
or \textit{QAP-norm}, with respect to $X$ if there is a (non-linear) 
operator $\cD:\R^n\ra\R^n$ a strictly decreasing function 
$c:\R_+\ra\R_+$, and an increasing function $C:\N\ra\R$ with the 
following properties:

(i) $\sp{f,\cD f}\leq 1$ for every $f\in X$;

(ii) $\sp{f,\cD f}\geq c(\e)$ for every $f\in X$ with 
$\|f\|\geq\e$;

(iii) $\|\cD f_1\dots\cD f_K\|^*\leq C(K)$ for any 
functions $f_1,\dots,f_K\in X$.
\end{definition}

It will help to explain the terminology if we introduce another norm,
which we shall call $\|.\|_{BAC}$. It is given by the formula
$\|f\|_{BAC}=\max\{|\sp{f,\cD g}|:g\in X\}$. The letters ``BAC'' stand
for ``basic anti-uniform correlation'' here. We shall call the functions 
$\cD f$
with $f\in X$ basic anti-uniform functions, and assume for convenience
that they span $\R^n$, so that $\|.\|_{BAC}$ really is a norm. Of
course, this norm depends on $X$ and $\cD$, but we are suppressing the
dependence in the notation.

By Lemma \ref{convexhulldual} the dual of the norm $\|.\|_{BAC}$ is 
given by the formula
\begin{equation*}
\|f\|_{BAC}^*=\inf\{\sum_{i=1}^k|\lambda_i|:f=\sum_{i=1}^k\lambda_i\cD f_i,
\ f_1,\dots,f_k\in X\}.
\end{equation*}
Thus, it measures the ease with which a function can be decomposed
into a linear combination of basic anti-uniform functions.
In terms of this norm, property (ii) above is telling us that if 
$f\in X$ and $\|f\|\geq\e$ then $\|f\|_{BAC}\geq c(\e)$. This
expresses a rough equivalence between the two norms, of a similar
kind to the rough equivalence between $\|f\|_{U^2}$ and $\|\hf\|_\infty$
when $\|f\|_\infty\leq 1$. It can also be thought of as a soft
inverse theorem; property (iii) then tells us that the functions
that we obtain from this inverse theorem have products that are
not too big.


Now let us briefly see why Theorem \ref{transference} generalizes
easily from preduals of algebra norms to QAP-norms. The following
result is not quite the result alluded to in the title of this
subsection, but it \textit{is} an abstract principle that can
be used as part of the proof the Green-Tao theorem. As we mentioned
in the introduction, the proof given here is much shorter and simpler
than the proof given by Green and Tao. (This is not quite trivial
to verify as they do not explicitly state the result, but the proof
here can be used to simplify Section 6 of their paper slightly, and
to replace Sections 7 and 8 completely.)

As a first step, we shall generalize Lemma \ref{polyapprox1}, the
simple result about polynomial approximations. The generalization is
equally straightforward: the main difference is merely that we need a
modificiation of the definition of the polynomial $R_P$.  Let us
suppose that $\|.\|$ is a QAP-norm and let $C:\N\ra\R$ be the function
given in property (iii) of that definition. If $P$ is the polynomial
$p(x)=a_nx^n+\dots+a_1x+a_0$, then we define $R_P'$ to be the
polynomial $C(n)|a_n|x^n+\dots+C(1)|a_1|x+|a_0|$: that is, we replace
the $k$th coefficient of $P$ by its absolute value and multiply it by
$C(k)$. If $J:\R\ra\R$ is a continuous function, and $C_1$, $C_2$ and
$\d$ are positive real numbers, we now define $\rho'(C_1,C_2,\d,J)$ to
be twice the infimum of $R_P'(C_2)$ over all polynomials $P$ such that
$|P(x)-J(x)|\leq\d$ for every $x\in[-C_1,C_1]$.

\begin{lemma}\label{polyapprox2}
Let $\|.\|$ be a QAP-norm, let $J:\R\ra\R$ be a continuous function
and let $C_1$, $C_2$ and $\d$ be positive real numbers, with
$C_1=C(1)C_2$. Then there exists a polynomial $P$ such that
$\|P\phi-J\phi\|_\infty\leq\d$ and
$\|P\phi\|^*\leq\rho'(C_1,C_2,\d,J)$ for every $\phi\in\R^n$ such that
$\|\phi\|_{BAC}^*\leq C_2$.
\end{lemma}

\begin{proof}
It is immediate from the definition of $\rho'(C_1,C_2,\d,J)$ that for every
$C_1$, $C_2$ and $\d$ there exists a polynomial $P$ such that
$|P(x)-J(x)|\leq\d$ for every $x\in[-C_1,C_1]$, and such that
$R_P'(C_2)\leq\rho'(C_1,C_2,\d,J)$.

Next, observe that if $X$ is the set specified in the definition of
QAP-norms, and $f$ is a function in $X$, then $\|\cD f\|^*\leq C(1)$,
by property (iii). Therefore, if $\phi\in\R^n$ is a function with
$\|\phi\|_{BAC}^*\leq C_2$, it follows that $\|\phi\|^*\leq
C(1)C_2=C_1$.  Then $\|\phi\|_\infty\leq C_1$ as well, from the
definition of QAP-norms. Since $P$ and $J$ agree to within $\d$ on
$[-C_1,C_1]$, it follows that $\|P\phi-J\phi\|_\infty\leq\d$.

From the formula for $\|\phi\|_{BAC}^*$ and the fact that this is at
most $C_2$ it follows that for any $\e>0$ we can write $\phi$ as a
linear combination of basic anti-uniform functions, with the absolute
values of the coefficients adding up to at most $C_2+\e$. Therefore,
for any $\e>0$ we can write $\phi^m$ as a linear combination of
products of $m$ basic anti-uniform functions, with the absolute values
of the coefficients adding up to at most $C_2^m+\e$. Each of these
products has $\|.\|^*$-norm at most $C(m)$, by property (iii) of
QAP-norms. Hence, by the triangle inequality, $\|\phi^m\|^*\leq
C(m)C_2^m$.  More generally, if $P$ is the polynomial 
$P(x)=a_nx^n+\dots+a_1x+a_0$, then by the triangle inequality 
we obtain that 
\begin{eqnarray*}
\|P\phi\|^*&\leq& |a_n|\|\phi^n\|^*+\dots+|a_1|\|\phi\|^*+|a_0|\\
&\leq& C(n)|a_n|C_2^n+\dots+C(1)|a_1|C_2+|a_0|\\
&=&R_P'(C_2).\\
\end{eqnarray*}
As we remarked at the beginning of the proof, this is at most
$\rho'(C_1,C_2,\d,J)$, so the lemma is proved.
\end{proof}

\begin{theorem}\label{transference2} 
Let $\mu$ and $\nu$ be non-negative functions on $\{1,2,\dots,n\}$ 
such that $\|\mu\|_1$ and $\|\nu\|_1$ are both at most 1, and
let $\eta,\d>0$. Let $\|.\|$ be a QAP-norm on $\R^n$, with respect to
the set $X$ of all functions $f\in\R^n$ such that
$|f(x)|\leq\max\{\mu(x),\nu(x)\}$ for every $x$. Let $J:\R\ra\R$ be
the function given by $J(x)=(x+|x|)/2$ and let
$\e=\d/2\rho'(C(1)c(\eta)^{-1},c(\eta)^{-1},\d/4,J)$, where $\rho'$ is
defined as in the discussion just above. Suppose that
$\|\mu-\nu\|\leq\e$. Then for every function $f$ with $0\leq f\leq\mu$
there exists a function $g$ such that $0\leq g\leq\nu(1-\d)^{-1}$ and
$\|f-g\|\leq\eta$.
\end{theorem}

\begin{proof}
An equivalent way of stating the conclusion is that
$f=g+h$ with $0\leq g\leq\nu(1-\d)^{-1}$ and $\|h\|\leq\eta$. 
Since such an $h$ will belong to $X$, we know that a sufficient
condition for $\|h\|$ to be at most $\eta$ is that $\|h\|_{BAC}$
is at most $c(\eta)$ (since $c$ is strictly decreasing). Thus,
if the result is false, then we can find a functional
$\phi$ such that $\sp{f,\phi}>1$, but $\sp{g,\phi}\leq 1$ for
every $g$ such that $0\leq g\leq\nu(1-\d)^{-1}$, and 
$\|\phi\|_{BAC}^*\leq c(\eta)^{-1}$. For the rest of the proof,
we shall write $C_2$ for $c(\eta)^{-1}$.

As in the proof of Theorem \ref{transference}, the first condition on
$\phi$ is equivalent to the statement that $\sp{\nu,\phi_+}\leq 1-\d$,
and we still have that $\phi_+=J\phi$. Also, Lemma \ref{polyapprox2}
gives us a polynomial $P$ such that $\|P\phi-J\phi\|_\infty\leq\d/4$
and $\|P\phi\|^*\leq\rho'=\rho'(C(1)C_2,C_2,\d/4,J)$.

Since $\sp{\nu,\phi_+}\leq 1-\d$ and $\|\nu\|_1\leq 1$, it follows that 
$\sp{\nu,P\phi}\leq 1-3\d/4$. Since $\|P\phi\|^*\leq\rho'$ and 
$\|\mu-\nu\|\leq\e$, it follows that $\sp{\mu,P\phi}\leq 1-3\d/4+\e\rho'$.
Since $\|\mu\|_1\leq 1$, it follows that $\sp{\mu,\phi_+}\leq 1-\d/2+\e\rho'$.
Since $f\leq\mu$ it follows that $\sp{f,\phi_+}\leq 1-\d/2+\e\rho'$,
and since $f\geq 0$ it follows that $\sp{f,\phi}\leq 1-\d/2+\e\rho'$,
which is a contradiction.
\end{proof}

The abstract theorem stated and proved by Tao and Ziegler is both
more and less general than Theorem \ref{transference2}. It is less
general in that it takes $\nu$ to be the uniform probability measure
(and uses the letter $\nu$ instead of $\mu$, so that the two measures
are $\nu$ and $\mathbf{1}$). But in a small way it is more general:
they observe that we did not really need the full strength of the
assumptions we made.

\subsection{Arithmetic progressions in the primes.}

In this section we shall briefly describe how a special case of
Theorem \ref{transference2}, the second transference principle we
proved earlier in the paper, was used by Green and Tao to prove that
the primes contain arbitrarily long arithmetic progressions.

The main idea of their proof is an ingenious way of getting round the
difficulty that the primes less than $N$ do not form a dense subset of
$\{1,2,\dots,N\}$. This sparseness problem occurs in several places in
the literature, and there is a method by which one can sometimes deal
with it, which is to exploit the fact that one has a lot of control
over \textit{random} (or random-like) sets. In particular, there are
various results that assert that if $X$ is a sparse random-like set
and $Y$ is a subset of $X$ that is dense in $X$ (in the sense that
$|Y|/|X|$ is bounded below by a positive constant) then $Y$ behaves in
a way that is analogous to how a dense set would behave. That is,
sparse sets can be handled if you can embed them densely into
random-like sets.

Green and Tao reasoned that an approach like this might work for the
primes. There is a standard technicality to deal with first, which is
that the primes are much denser in some arithmetic progressions than
others. A moment's thought shows that this makes it impossible to
embed the primes from 1 to $N$ densely into a quasirandom
set. However, one can restrict to an arithmetic progression in which
the primes are particularly dense (by looking at primes that are
congruent to $a$ mod $m$, where $m$ is the product of the first few
primes and $a$ is coprime to $m$), in which this problem effectively
disappears.

To carry out their approach, they needed to do two things. First, they
had to prove that there was indeed a quasirandom set containing the
primes (inside a suitable arithmetic progression, but we'll use the
word ``primes" as a shorthand here) that was not much bigger than the
primes. If they could do that, then the general principle that
relatively dense subsets of quasirandom sets behave like dense sets
would suggest that the primes should behave like a dense set. Since
dense sets contain plenty of arithmetic progressions, so should the
primes. The second stage of their proof was to make this heuristic
argument rigorous.

As it turns out, they did not construct a quasirandom superset of the
primes, but an object that they called a \textit{pseudorandom
measure}. This was a non-negative function $\nu$ that did not have to
be 01-valued, but in other respects behaved like a superset of the
primes. (In fact, they normalized it to have average 1, but even then
it did not take just one non-zero value.) The construction of $\nu$
was based on very recent (at the time) results of Goldston and
Y{\i}ld{\i}r{\i}m \cite{goldyild}. This part of the proof belongs squarely in
analytic number theory and we shall say no more about it here.

The other part of the proof proceeded as follows. Let $\nu$ be a
pseudorandom measure: that is, a non-negative function defined on
$\{1,2,\dots,N\}$ such that $\|\nu\|_1=1$, which satisfied certain
quasirandom properties. (These properties were similar to, but
stronger than, the assertion that $\|\nu-\mathbf{1}\|_{U^k}$ was very
small.) Let us call a set $A$ \textit{dense relative to} $\nu$ if
there is a positive constant $\lambda$ such that $\lambda A\leq\nu$
and $\|\lambda A\|_1\geq c$ for some positive constant $c$ that does
not depend on $N$. Since $\|\nu-\mathbf{1}\|_{U^k}$ is small, the
transference principle of Theorem \ref{transference2} can be used to
replace the function $\lambda A$ by a function $f$ that takes values
in $[0,1]$ and has the property that $\|f-\lambda A\|_{U^k}$ is small,
provided, that is, that the hypotheses of Theorem \ref{transference2}
are satisfied.

The programme for completing the proof is therefore clear: one must
prove that the hypotheses are indeed satisfied, and one must prove
that the fact that $\|f-\lambda A\|_{U^k}$ is small allows us to
conclude that $A$ contains arithmetic progressions of length $k+2$ (as
one expects, since in other contexts the $U^k$ norm controls
progressions of this length).

Let us briefly recall what these hypotheses are. We define $X$ to be
the set of all functions that are bounded above in modulus by $\nu+1$,
and we would like the $U^k$ norm to be a QAP-norm with respect to
$X$. (These were defined at the beginning of Section \ref{GTZ}.) Not
surprisingly, as our non-linear operator $\cD$ we take the operator
defined just before Proposition \ref{u3softinverse} (for the
appropriate $k$), except that for convenience we multiply it by
$2^{-(k+1)}$.

The first hypothesis is that $\sp{f,\cD f}\leq 1$ for every $f\in
X$. It is straightforward to check from Green and Tao's definition of
pseudorandomness that $\|f\|_{U^k}$ is at most $2^k+o(1)$ for each
$f\in X$, and therefore this hypothesis is satisfied.

The second is that $\sp{f,\cD f}\geq c(\e)$ for every $f\in X$ with
$\|f\|_{U^k}\geq\e$.  But this is true because, with our definition of
$\cD$, $\sp{f,\cD f}=2^{-(k+1)}\|f\|_{U^k}$.  (We have essentially
given this argument already, in Proposition \ref{u3softinverse}.)

The third is that products of basic anti-uniform functions have
bounded $(U^k)^*$-norms.  This is a lemma of Green and Tao that we
stated as Lemma \ref{bauproducts}. It should be noted that to prove
this they required quasirandomness hypotheses on $\nu$ that are
stronger than one might expect: in particular they needed more
than just that $\nu$ should be close to $\mathbf{1}$ in some
$U^r$ norm. (The precise condition they needed is called
the \textit{correlation condition} in their paper.) It is 
not known whether there exists an $r$ such that their transference
theorem holds under the hypothesis that $\|\nu-\mathbf{1}\|_{U^r}$
is small.

The one remaining ingredient of their argument is what they call a
``generalized von Neumann theorem," in which they establish the fact
mentioned above, that if $\|f-\lambda A\|_{U^k}$ is small then $A$
contains arithmetic progressions of length $k+2$. More precisely,
\begin{equation*}
\lambda^{k+2}\E_{x,d}A(x)A(x+d)\dots A(x+(k+1)d)\approx
\E_{x,d}f(x)f(x+d)\dots f(x+(k+1)d).
\end{equation*}
If $A$ is a dense set, so that $\lambda$ is bounded above by a constant independent
of $N$, then this is a standard result, but it is quite a bit harder to prove when all one 
knows about $A$ is that $\lambda A$ is bounded above by a pseudorandom measure.




\section{Tao's structure theorem.}\label{taostruct}

In this section we shall combine some of the methods and results of
previous sections in order to obtain a general structure theorem for
bounded functions. This result resembles Proposition
\ref{strongdecomposition} in that we decompose a function $f$ as a sum
$f_1+f_2+f_3$ with $\|f_1\|^*\leq C$, $\|f_2\|\leq\eta(C)$ and
$\|f_3\|_2\leq\e$, but this time we shall assume that $f$ takes values
in an interval $[a,b]$ and deduce stronger properties of the functions
$f_i$: in particular, $f_1$ will also take values in the interval
$[a,b]$. In order to do this, we shall need to use polynomial
approximations.  It would be possible to prove a result about
QAP-norms, but the notation is simpler if we assume the stronger
hypothesis that the dual norm $\|.\|^*$ is an algebra norm. As we
shall see, this result is general enough to apply in many interesting
situations. 

Here, then, is the structure theorem we shall prove in this section.
Tao's structure theorem is essentially the same result, but for a
specific sequence of algebra norms. However, his method can easily
be modified to prove this more general formulation. (In other words, 
the point of this section is the method of proof rather than the
extra generality of the conclusion.) We should mention here that 
there are other results of a similar flavour to Tao's, which are
often referred to as ``arithmetic regularity lemmas''. The following
result can be thought of as an abstract arithmetic regularity lemma.

\begin{theorem}\label{taostructure}
 Let $\|.\|$ be a norm defined on $\R^n$, and suppose that the dual
norm $\|.\|^*$ is an algebra norm. Let $f\in\R^n$ be a function that
takes values in the interval $[a,b]$.  Let
$\eta:\mathbb{R}_+\rightarrow\mathbb{R}_+$ be a positive decreasing
function and let $\e>0$. Then there is a constant $C_0$, depending
on $\eta$ and $\e$ only, such that $f$ can be written
as a sum $f_1+f_2+f_3$, with $\|f_1\|^*\leq C_0$,
$\|f_2\|\leq\eta(\|f_1\|^*)$, and $\|f_3\|_2\leq\e$.  Moreover, $f_1$
and $f_1+f_3$ both take values in $[a,b]$.
\end{theorem}

The last condition may look slightly strange, but it is important in
applications. For instance, for Tao's application to Szemer\'edi's
theorem, $[a,b]$ is the interval $[0,1]$, and $f_1$ is the
``structured part'' of $f$. The key step in his argument is that $\E_x
f_1\geq\d$ implies that $\E_{x,d}f_1(x)f_1(x+d)\dots f_1(x+(k-1)d)\geq
c(\d)>0$, and more generally that the same is true of $f_1+f_3$: that
is, after a small $L_2$-perturbation of the function $f_1$. However,
$c(\d)$ is much smaller than $\d$; as a result, it is crucial that
both $f_1$ and $f_1+f_3$ should be positive, so that $c(\d)$ is not
swamped by a negative error term.

There is a simple way of making Theorem \ref{taostructure} more
general, and this is very important for some applications, including
Tao's application to Szemer\'edi's theorem. In order to explain the
generalization, it will be convenient to introduce another definition.

\begin{definition} Let $\|.\|$ and $|.|^*$ be two norms on $\R^n$
and let $c:(0,1]\rightarrow(0,1]$ be a strictly increasing function.
Then $|.|^*$ is an \emph{approximate dual} (at rate $c$) for $\|.\|$ 
if the following two conditions hold:

(i) $\sp{f,\phi}\leq\|f\||\phi|^*$ for any two functions $f$ and
$\phi$ in $\R^n$;

(ii) if $\|f\|_\infty\leq 1$ and $\|f\|\geq\e$ then there exists
$\phi\in\R^n$ such that $|\phi|^*\leq 1$ and $\sp{f,\phi}\geq c(\e)$.
(Equivalently, $|f|\geq c(\e)$, where $|.|$ is the predual of $|.|^*$.)
\end{definition}

The first of these estimates is equivalent to the assertion that
$\|\phi\|^*\leq|\phi|^*$ for every $\phi\in\R^n$. The second is
equivalent to the assertion that if $\|f\|_\infty\leq 1$, then
$|f|\geq c(\|f\|)$. Therefore, if a norm $\|.\|$ merely has an
\textit{approximate} dual $|.|^*$ that is an algebra norm, we can
apply Theorem \ref{taostructure} to the norm $|.|$ and conclude that
$|f_1|^*\leq C_0$, $\|f_2\|\leq c^{-1}(\eta(|f_1|^*))$ and
$\|f_3\|\leq\e$. Since $\eta$ can be chosen to tend to zero
arbitrarily fast, so can $c^{-1}\circ\eta$. Thus, Theorem
\ref{taostructure} has the following immediate corollary.

\begin{corollary}\label{taostructure2}
Let $\|.\|$ and $|.|$ be two norms on $\R^n$ and suppose that $|.|^*$
is an approximate dual for $\|.\|$. Let $f$ be a function that takes
values in an interval $[a,b]$. Let
$\eta:\mathbb{R}_+\rightarrow\mathbb{R}_+$ be a positive decreasing
function and let $\e>0$. Then there is a constant $C_0$, depending
only on $\eta$, $\e$ and the function $c$ that appears in the
specification of the approximate duality, such that $f$ can be written
as a sum $f_1+f_2+f_3$, with $|f_1|^*\leq C_0$,
$\|f_2\|\leq\eta(\|f_1\|^*)$, and $\|f_3\|_2\leq\e$.  Moreover, $f_1$
and $f_1+f_3$ both take values in $[a,b]$.
\end{corollary} 

\subsection{A proof of the structure theorem.}

The proof we shall give in this paper is quite different from that of
Tao. The main idea is to start with a decomposition obtained using
Proposition \ref{strongdecomposition} (which was an easy consequence
of the Hahn-Banach theorem) and to adjust it until the functions $f_1$
and $f_1+f_3$ have the right ranges. During the process of adjustment,
we shall have cause to use Theorem \ref{transference}, the first of
the transference theorems obtained in the previous section. The proof
is conceptually very simple, but it involves a longish sequence of
small calculations to check that the errors that we introduce when we
adjust our decomposition are small.

To begin with, then, let $\theta$ be a decreasing positive function
and $\b$ a positive constant, both to be specified later, and apply
Proposition \ref{strongdecomposition} to write $f$ as $f_1+f_2+f_3$
with $\|f_1\|^*=K$, $\|f_2\|\leq\theta(K)$ and $\|f_3\|\leq\b$. 
Here, $K$ is bounded above by a function of $\theta$ and $\b$, 
so later we shall need $\theta$ and $\b$ to depend only on $\eta$
and $\e$.

We would now like to modify $f_1$ so that it takes values in the
interval $[a,b]$. The obvious way of doing this is to apply Lemma
\ref{polyapprox1} with the continuous function $J$ that takes the
value $a$ when $x<a$, $b$ when $x>b$ and $x$ when $x\in [a,b]$.  This
gives us a new function $Pf_1$ such that $\|Pf_1-Jf_1\|_\infty\leq\d$ and
$\|Pf_1\|^*\leq\rho=\rho(K,\d,J)$. The first inequality implies that
$Pf_1$ takes values in $[a-\d,b+\d]$, and a small adjustment will
correct that to $[a,b]$. However, before we do the adjustment, let us
check that $f_1-Pf_1$ is small in an appropriate sense. Intuitively,
this is plausible: it should not be possible for the structured part
of $f$ to stray too far from the interval $[a,b]$ for too long.  This
intuition turns out to be correct, and proving it rigorously is not
very hard.

\begin{lemma}
Let $f_1$ and $Pf_1$ be the functions just defined. Then provided that
the function $\theta$ is sufficiently small (in terms of $a$, $b$ and
$\b$), we have the inequality $\|f_1-Pf_1\|_2\leq 3\b/2$.
\end{lemma}

\begin{proof}
From the decomposition $f=f_1+f_2+f_3$ we obtain the decomposition
\begin{equation*}
f_1-Pf_1=(f-Pf_1)-f_2-f_3
\end{equation*}
We shall now bound $\|f_1-Pf_1\|_2^2$ by looking at the inner products
of $f_1-Pf_1$ with each of the three terms on the right-hand side.

First of all, if $f_1(x)>b$ then $Jf_1(x)=b$, so $|Pf_1(x)-b|\leq\d$,
and therefore $f_1(x)-Pf_1(x)\geq -\d$. Since $f$ takes values in
$[a,b]$, we also find that $f(x)-Pf_1(x)\leq\d$. Similarly, if
$f_1(x)<a$ then we find that $f_1(x)-Pf_1(x)\leq\d$ and
$f(x)-Pf_1(x)\geq-\d$. If $f_1(x)\in[a,b]$, then $f_1(x)=Jf_1(x)$, so
$|f_1(x)-Pf_1(x)|\leq\d$, and $|f(x)-Pf_1(x)|\leq (b-a+2\d)$. It
follows from these three estimates that $\sp{f_1-Pf_1,f-Pf_1}\leq
4\d^2+\d(b-a)$.

Since $\|f_1\|^*\leq K$, $\|Pf_1\|^*\leq\rho$, and $\|f_2\|\leq\theta(K)$, 
it follows that 
\begin{equation*}
|\sp{f_1-Pf_1,f_2}|\leq (K+\rho)\theta(K).
\end{equation*}

For the third inner product we use Cauchy-Schwarz to give a trivial 
implicit estimate: 
\begin{equation*}
|\sp{f_1-Pf_1,f_3}|\leq\b\|f_1-Pf_1\|_2.
\end{equation*}

From the estimates for these inner products it follows that
\begin{equation*}
\|f_1-Pf_1\|_2^2\leq 4\d^2+\d(b-a)+(K+\rho)\theta(K)
+\b\|f_1-Pf_1\|_2.
\end{equation*}
Therefore, if we choose $\d$ such that $4\d^2+\d(b-a)\leq\b^2/4$ and
$\theta$ in such a way that 
$(K+\rho(K,\d,J))\theta(K)\leq\b^2/2$ for every $K$, then 
\begin{equation*}
\|f_1-Pf_1\|_2^2\leq \b^2/4+\b^2/2+\b\|f_1-Pf_1\|_2,
\end{equation*}
from which it follows, on completing the square, that 
$\|f_1-Pf_1\|_2\leq 3\b/2$, as claimed. To complete the proof, note 
that the condition on $\theta$ depends on $\rho$, and hence on 
$\d$, and $\d$ depends on $a$, $b$ and $\b$.
\end{proof}


The next step is very simple. Let $L$ be the linear function that
takes $a-\d$ to $a$ and $b+\d$ to $b$. Then $|L(x)-x|$ is at most
$\d$ for every $x$ in the interval $[a-\d,b+\d]$. Since $f_1$ takes
values in this interval, it follows that $\|LPf_1-Pf_1\|_\infty\leq\d$.
Also, if we write $L(x)=\lambda x+\mu$, it is easy to see that 
$0<\lambda<1$, from which it follows that $\|LPf_1\|^*\leq \|Pf_1\|^*+\mu$
(since $|\mathbf{1}|^*=1$). A small calculation shows that 
$\mu=-(a+b)\d/(a-b-2\d)$, so $\|LPf_1\|^*\leq 2\rho$,
provided $\d$ is moderately small (depending on $a$ and $b$).

Let us now see where we have reached. We started with a decomposition
$f=f_1+f_2+f_3$, and we have now modified $f_1$, first to $Pf_1$ and
then to $LPf_1$. The first modification incurred an extra error of
$L_2$-norm at most $3\b/2$, and the second an extra error of
$L_\infty$-norm, and hence $L_2$-norm, at most $\d$. If we assume that
$\d\leq\b/2$ then we find that we have a decomposition
$f=g_1+g_2+g_3$, where $g_1=LPf_1$, $g_2=f_2$, and
$g_3=f_3+(f_1-LPf_1)$.  We have shown that $\|g_1\|^*\leq 2\rho$, that
$\|g_2\|=\|f_2\|\leq\theta(K)$, and that
$\|g_3\|_2\leq\b+3\b/2+\b/2=3\b$.  Moreover, $g_1$ takes values in the
interval $[a,b]$.

This gives us most of what we want (if we choose $\b$ and $\theta$
appropriately). The main thing we are missing is any information
about the range of $g_1+g_3$. In order to obtain the extra property
that $g_1+g_3$ takes values in $[a,b]$, we shall focus on the
equivalent problem of ensuring that $f(x)-b\leq g_2(x)\leq f(x)-a$ for
every $x$.

Note that $f(x)-b\leq 0$ and $f(x)-a\geq 0$ for every $x$. Our strategy 
for obtaining these bounds on $g_2$ is even simpler than our strategy
for adjusting $f_1$ earlier: we shall replace $g_2(x)$ by $f(x)-a$
whenever $g_2(x)>f(x)-a$, and similarly on the other side. However,
if that is all we do then we lose all information about $\|g_2\|$.
This is where Theorem \ref{transference}, our first transference theorem,
comes in: when we adjust the positive part of $g_2$ we can use 
Theorem \ref{transference} to make a complementary adjustment to the
negative part, and vice versa. 

Let us therefore set $g_2'(x)$ to be $\min\{g_2(x),f(x)-a\}$ for each $x$. 
First we need a simple lemma.

\begin{lemma}\label{estimate1}
If $g_2'=\min\{g_2,f-a\}$, then $\|g_2-g_2'\|_2\leq 3\b.$
\end{lemma}

\begin{proof} 
For every $x$, either $g_2(x)-g_2'(x)=0$ or 
\begin{equation*}
0\leq g_2(x)-g_2'(x)=g_2(x)-f(x)+a=a-g_1(x)-g_3(x)\leq -g_3(x),
\end{equation*}
where the last inequality follows from the fact that $g_1(x)\in[a,b]$
for every $x$. It follows that $\|g_2-g_2'\|_2\leq\|g_3\|_2$, which we
have established to be at most $3\b$.
\end{proof}

Our first attempt at adjusting the decomposition is to write
\begin{equation*}
f=g_1+g_2'+(g_3+g_2-g_2').
\end{equation*}
Our main problem now is that we do not have a good estimate for
$\|g_2'\|$. To deal with this, we shall adjust the negative part of
$g_2$ as well, using Theorem \ref{transference}. Let $\mu=(g_2)_+$
and let $\nu=(g_2)_-$. Then $\mu$ and $\nu$ are disjointly supported,
so both $\|\mu\|_1$ and $\|\nu\|_1$ are at most $\|g_2\|_1$, which
is at most $\|g_2\|_2$. Since $g_2=(f-g_1)+g_3$ and $f$ and $g_1$ take
values in $[a,b]$, $\|g_2\|\leq|b-a|+3\b$, by the triangle inequality
and our estimate for $\|g_3\|_2$. Let $\a=|b-a|+3\b$.

We now apply Theorem \ref{transference} with $\mu$ and $\nu$ as above
and with $f=(g_2')_+$. Strictly speaking, this is not quite accurate,
since the upper bounds for $\|\mu\|_1$ and $\|\nu\|_1$ are $\a$ rather
than 1, but we can look at the functions $\a^{-1}\mu$, $\a^{-1}\nu$
and $\a^{-1}f$ instead. The main hypothesis we have is that
$\|g_2\|=\|\mu-\nu\|\leq\theta(K)$, so we can take $\e$ to be
$\a^{-1}\theta(K)$ in Theorem \ref{transference}. If $\tau>0$ is a
constant such that $\a^{-1}\theta(K)=\d/2\rho(\a\tau^{-1},\b/4,J)$
(where now $J(x)$ is the function $(x+|x|)/2$), then we may conclude
that there is a function $g$ such that $0\leq g\leq\nu(1-\b)^{-1}$ and
$\|f-g\|\leq\tau$. The important thing to note here is that $\tau$
tends to zero as $\theta(K)$ tends to zero.

Define $g_2''$ to be $f-g$. This gives us a decomposition 
\begin{equation*}
f=g_1+g_2''+[g_3+(g_2-g_2')+(g_2'-g_2'')].
\end{equation*}
We have the upper bounds $g_2''(x)\leq f(x)-a$ for every $x$, 
and $\|g_2''\|\leq\tau$. However, we have not yet checked that 
$\|g_2'-g_2''\|_2$ is small. For this we need another simple lemma.

\begin{lemma}\label{estimate2}
Let $\nu\in\R^n$ be a non-negative function and suppose that $\nu$ 
can be written as a sum $\nu_1+\nu_2$, where $\|\nu_1\|_\infty\leq\a$
and $\|\nu_2\|_2\leq\g$. Then $\|h\|_2\leq\g+(\a\|h\|_1)^{1/2}$ for 
any function $h$ with $0\leq h\leq\nu$.
\end{lemma}

\begin{proof}
By the positivity of $h$ and $\nu$, 
\begin{equation*}
\|h\|_2^2\leq\sp{h,\nu_1+\nu_2}\leq\a\|h\|_1+\g\|h\|_2.
\end{equation*}
The bound stated is an easy consequence of this.
\end{proof}

\begin{corollary}\label{estimate3}
Let $g''$ be any function such that $g''(x)=g_2'(x)$ when $g_2'(x)$
is non-negative, and $0\geq g''(x)\geq g_2'(x)$ otherwise. Suppose
also that $\|g_2\|\leq\tau$. Then 
$\|g''-g_2'\|_2\leq 3\b+(\a(\tau+3\b))^{1/2}$.
\end{corollary}

\begin{proof}
It follows from the hypotheses that $0\leq g''(x)-g_2'(x)\leq\nu(x)$
for every $x$.  Recall that $g_2=(f-g_1)+g_3$ and that
$\|f-g_1\|_\infty\leq b-a$.  It follows easily that $\nu=(g_2)_-$
satisfies the conditions of Lemma \ref{estimate2}, with
$\g=3\b$. (We could improve $\a$ to $b-a$, but this is not worth
bothering about.)

Applying the lemma, we deduce that $\|g''-g_2'\|_2\leq
3\b+(\a\|g''-g_2'\|_1)^{1/2}$. Now let us turn our attention to
bounding $\|g''-g_2'\|_1$. Since $\|f-g\|\leq\tau$ and $\|.\|^*$ is an
algebra norm, it follows from Lemma \ref{algebrabasics} that
$|\E_x(f(x)-g(x))|\leq\tau$.  But
$|\E_x(f(x)-g(x))|\geq\|g''-g_2'\|_1-\|g_2-g_2'\|_1$, since
$g''\geq g_2'$, and $\|g_2-g_2'\|_1\leq\|g_2-g_2'\|_2$, which we
have already shown is at most $3\b$. Therefore,
$\|g''-g_2'\|_1\leq\tau+3\b$. Inserting this bound into the
estimate at the beginning of this paragraph, we find that
$\|g''-g_2'\|_2\leq 3\b+(\a(\tau+3\b))^{1/2}$, as claimed. 
\end{proof}

Since the function $g_2''$ constructed earlier satisfies the
hypotheses required of $g''$ in Corollary \ref{estimate3}, we now have
an improved decomposition $f=h_1+h_2+h_3$, where $h_1=g_1$,
$h_2=g_2''$ and $h_3=g_3+g_2-g_2''$. Our arguments so far have shown
that $h_1$ takes values in $[a,b]$, that $\|h_1\|^*\leq 2\rho$, that
$\|h_2\|\leq\tau$, that $h_2(x)\leq f(x)-a$ for every $x$, and that
$\|h_3\|_2\leq 6\b+3\b+(\a(\tau+3\b))^{1/2}$.  If $\b$ is sufficiently
small (depending on $b-a$ if that is small, which in a typical
application it will not be), and $\tau$ is sufficiently small
(depending on $\b$), then this last quantity is at most
$2(\b(b-a))^{1/2}$, which we shall call $\zeta$.

From the way we constructed $h_2$, we know that the sign of $h_2$
is the same as that of $g_2$, and that $|h_2(x)|\leq |g_2(x)|$
for every $x$. We have obtained the upper bound of $f-a$ that we 
wanted for $h_2$; now we need a further adjustment in order to 
obtain a lower bound of $f-b$. It is obvious how to do this:
we shall sketch the argument only very briefly. 

First, we let $h_2'(x)=\max\{h_2(x),f(x)-b\}$ for every $x$. Then a
simple modification of Lemma \ref{estimate1} shows that
$\|h_2-h_2'\|_2\leq 3\zeta$.

Next, we use Theorem \ref{transference} to reduce the positive part
of $h_2'$, while leaving the negative part unchanged, to create a 
function $h_2''$ with $\|h_2''\|$ small. If we let $\alpha'=b-a+3\zeta$,
then the same argument as before gives us an upper bound 
$\|h_2''\|\leq\kappa$, where $\kappa$ is a constant such that 
$\alpha'^{-1}\tau=\d/2\rho(\a'\kappa^{-1},\zeta/4,J)$. In particular,
$\kappa$ tends to zero as $\tau$ tends to zero.

Next, a simple modification of Corollary \ref{estimate3} tells us that
$\|h_2''-h_2'\|_2\leq 3\zeta+(\a'(\kappa+3\zeta))^{1/2}$.  Therefore,
we have a decomposition $f=u_1+u_2+u_3$, with $u_1=h_1$, $u_2=h_2''$
and $u_3=h_3+h_2-h_2''$. Since $u_1=h_1$, it takes values in
$[a,b]$. The construction of $h_2''$ guarantees that $f(x)-b\leq
u_2(x)\leq f(x)-a$, and hence that $u_1+u_3$ takes values in
$[a,b]$. Finally, we have the estimates $\|u_1\|^*\leq 2\rho$,
$\|u_2\|\leq\kappa$, and $\|u_3\|_2\leq
6\zeta+3\zeta+(\a'(\kappa+3\zeta))^{1/2}$. If $\zeta$ is small enough
(depending on $b-a$) and $\kappa$ is small enough (depending on
$\zeta$), then this last quantity is at most $2((b-a)\zeta)^{1/2}$.

Now let us see why these estimates are enough, recalling from the
beginning of the proof that we are free to choose $\b$ and
$\theta$. To begin with, we need $2((b-a)\zeta)^{1/2}$ to be at most
$\e$. But $\zeta$ tends to zero with $\b$, so this is easily achieved.
Next, recall that $\rho=\rho(K,\d,J)$. We would like $\kappa$ to be at
most $\eta(\rho)$, which we shall ensure by making a suitable choice
of $\theta$. The constant $\d$ depends on $\b$, $a$ and $b$ only,
while $\kappa$ tends to zero with $\tau$, which tends to zero with
$\theta(K)$. Thus, for each $K$ we can choose $\theta(K)$ in a way
that depends on $K$, $\b$, $a$ and $b$ only, such that
$\kappa\leq\eta(\rho)$. The proof is complete.

\subsection{Decomposition theorems with bounds on ranges.}

As a simple application of Theorem \ref{taostructure}, we shall now
obtain the improvement that we promised earlier to our results
about deducing decomposition theorems from inverse theorems. So far,
we have shown that a function can be decomposed into a multiple of a
convex combination of structured functions, plus an error, provided
that we have a suitable inverse theorem concerning the structured
functions and the kind of error we are prepared to allow. As we
commented, it is sometimes useful to obtain a decomposition for which
the ``structured part" is bounded. We shall see that Theorem
\ref{taostructure} implies rather easily that such a decomposition
exists, and it has the added advantage of yielding an $L_2$ error term
rather than the $L_1$ error term that appears in Theorem
\ref{decompositiontheorem} or weaker theorems of a similar type that
were discussed earlier in Section \ref{decompositions}. However, the
bound on the sum of the coefficients of the structured functions is
very bad. For some applications, this is not a concern, but for others
it turns out to be preferable to use weaker theorems.

\begin{theorem}
Let $\|.\|$ be a norm on $\R^n$ and let $\Phi\subset\R^n$ be a set of
functions satisfying the following properties
for some strictly increasing function $c:(0,1]\ra(0,1]$:

(i) $\Phi$ contains the constant function $\mathbf{1}$, $\Phi=-\Phi$,
$\|\phi\|_\infty\leq 1$ for every $\phi\in\Phi$, and the linear span of 
$\Phi$ is $\R^n$;

(ii) $\sp{f,\phi}\leq 1$ for every $f$ with $\|f\|\leq 1$ and
every $\phi\in\Phi$;

(iii) if $\|f\|_\infty\leq 1$ and $\|f\|\geq\e$ then there exists
$\phi\in\Phi$ such that $\sp{f,\phi}\geq c(\e)$.

Let $\e>0$ and let $\eta:\R_+\ra\R_+$ be a strictly decreasing
function. Then there is a constant $M_0$, depending only on $\e$ and
the functions $c$ and $\eta$, such that every function $f\in\R^n$ that
takes values in $[0,1]$ can be decomposed as a sum $f_1+f_2+f_3$, with
the following properties: $f_1$ and $f_1+f_3$ take values in $[0,1]$;
$f_1$ is of the form $\sum_i\lambda_i\psi_i$, where
$\sum_i|\lambda|=M\leq M_0$ and each $\psi_i$ is a product of
functions in $\Phi$; $\|f_2\|\leq\eta(M)$; $\|f_3\|_2\leq\e$.
\end{theorem}

\begin{proof}
Let $\Psi$ be the set of all products of functions in $\Phi$. Define a 
norm $|.|^*$ by taking $|g|^*$ to be the infimum of all sums 
$\sum_i|\lambda_i|$ such that $g$ can be written as $\sum_i\lambda_i\psi_i$
with every $\psi_i$ in $\Psi$. It is straightforward to check that
this is an algebra norm. (The fact that it is a norm rather than a 
seminorm relies on the boundedness of functions in $\Phi$, which one 
could in fact deduce from (ii) rather than stating as a separate
assumption.) Moreover, (ii) and (iii) imply easily that $|.|^*$ is
an approximate dual for $\|.\|$. Therefore, Corollary \ref{taostructure2}
implies the result.
\end{proof}

The following simple trick is important in applications. Property
(iii) in the statement of the theorem is the assertion that there is
an inverse theorem relating the norm $\|.\|$ to the set of functions
in $\Phi$. However, the conclusion of the theorem concerns
\textit{products} of functions in $\Phi$, and in practice it often
happens that the set $\Phi$ of functions that one obtains from an
inverse theorem is not closed under pointwise multiplication. However,
it also often happens that one can give an explicit description of
products of functions in $\Phi$, and that this description becomes
only gradually less useful as the number of functions in the product
increases. Under such circumstances, one can replace $\Phi$ by the set
$\{\mathbf{1},-\mathbf{1}\}\cup(\Phi/2)$. This modified set clearly
satisfies all the hypotheses that $\Phi$ was required to satisfy, but
now the corresponding set $\Psi$ comes with a ``penalty'' of $2^{-k}$
attached to a product of $K$ functions. This means that the sum of the
$|\lambda_i|$ over products of significantly more than $\log_2M_0$
functions in $\Phi$ make a very small (in $L_\infty$) contribution to
$f_1$ and can be absorbed into the error term.

\subsection{Applying Tao's structure theorem}

We shall not actually give applications of the structure theorem here,
but merely comment on how it is applied. The rough idea, as we have
already seen, is to express a bounded function (such as, for instance,
the characteristic function of a dense subset of $\mathbb{Z}_N$) as a
sum of a structured part, a quasirandom part, and an $L_2$ error.  To
do this, we need to choose a norm $\|.\|$ that measures
quasirandomness in a useful way, such that its dual norm $\|.\|^*$ is
an algebra norm with the property that if $\|\phi\|^*$ is bounded then
we ``understand" $\phi$ and can regard it as structured.  As we have
seen, a simple (but useful) example of such a norm is
$\|f\|=\|\hf\|_\infty$.

Let us briefly consider this example. If we have written a function
$f$ as $f_1+f_2+f_3$ in such a way that $\|\hf\|_1\leq C$,
$\|\hf\|_2\leq\eta(C)$ and $\|f_3\|_2\leq\e$, then we can analyse it
as follows.

We first show that $f_1$ is ``approximately smooth" in the following
sense. Let $\d,\theta>0$ be small constants to be chosen later, and
let $K$ be the set of all $r$ such that $|\hf_1(r)|\geq\d$.  Since
$\|\hf_1\|_2^2=\|f_1\|_2^2\leq 1$, it follows that
$|K|\leq\d^{-2}$. Now let $B$ be the set of all $x\in\Z_N$ such that
$|\omega^{rx}-1|\leq\theta$ for every $r\in K$. Sets like $B$ are
called \textit{Bohr neighbourhoods} and have many good properties, but
for now we remark merely that a fairly straightforward argument shows
that the cardinality of $B$ is at least $\theta^{|K|}N$.

Now let $\b$ be the \textit{characteristic measure} of $B$: that is,
the function that takes the value $N/|B|$ on $B$ and $0$
elsewhere. This multiple of $B$ is chosen so that $\|\b\|_1=1$. A
useful property of $\beta$ is that $f_1$ is close to $f_1*\b$ in
$L_2$. This can be shown with the help of Fourier transforms: the
general method is known as \textit{Bogolyubov's method}, and it is a
very useful tool in additive combinatorics.  We begin by observing
that
\begin{eqnarray*}
\|f_1-f_1*\b\|_2^2&=&\|\hf_1-\hf_1\hat{\b}\|_2^2\\ &=&\sum_{r\in
K}|\hf_1(r)|^2|1-\hat{\b}(r)|^2+\sum_{r\notin
K}|\hf_1(r)|^2|1-\hat{\b}(r)|^2\\
\end{eqnarray*}
For every $r\in K$ we have $|1-\hat{\b}(r)|=|\E_{x\in
B}(1-\omega^{rx})|\leq\theta$, so the first sum is at most
$\theta^2\|\hf\|_2^2\leq\theta^2$. We also have the trivial estimate
that $|1-\hat{\b}(r)|\leq 2$, so the second sum is at most
$4\d\|\hf\|_1\leq 4\d C$. Thus, by choosing $\d$ and $\theta$
appropriately, we can ensure that $f_1$ and $f_1*\b$ are close in
$L_2$, as claimed.

This tells us that for a typical pair $x$ and $y$, if $x-y\in B$, then
$f_1(x)$ and $f_1(y)$ are close. Equivalently, $f_1$ is almost always
roughly constant on translates of $B$.

Now if we choose $\eta(C)$ to be small enough, then $f_2$ is highly
quasirandom even compared with the size of $B$. That is,
$\|\hf_2\|_\infty$ is so small that even the restrictions of $f_2$ to
translates of $B$ behave quasirandomly (in a sense that one can make
precise in several natural ways). This means that even though $f_2$
may have a large $L_2$ norm, we may nevertheless think of $f_1+f_2$ as
a tiny perturbation of $f_1$. For instance, if a $x$ is a typical
element of $\mathbb{Z}_N$ and $f_1(x)\geq c$, then the smoothness of
$f_1$ guarantees that $f_1(y)\geq c/2$ for almost every $y\in
x+B$. From this and the positivity of $f_1$ it follows (if $B$
satisfies a certain technical condition that one can always ensure)
that
\begin{equation*}
\E_{x,d}f_1(x)f_1(x+d)f_1(x+2d)
\end{equation*}
is bounded below by some (very small) positive constant related to the
density of $B$, which depended on $C$ only. If $\eta(C)$ is much
smaller than this constant, then perturbing by $f_2$ cannot change
this lower bound to zero.

This is not quite a sketch proof of Roth's theorem (though it is
close), because there remains the problem of dealing with $f_3$. In
fact, the correct order to work in is to think about $f_1$ first, then
$f_1+f_3$, and finally $f_1+f_2+f_3$. This is why it is so helpful for
$f_1$ and $f_1+f_3$ to be non-negative functions.

The above idea can be thought of as a discrete analogue of at least
one ergodic-theoretic proof of Roth's theorem. Tao applied his
structure theorem to a sequence of cleverly constructed algebra norms
in order to extend the argument to a proof of the general case of
Szemer\'edi's theorem. Unfortunately, the analysis of the structured
function $f_1$ becomes far harder: that is where the real difficulty
of his argument lies.

The structure theorem can also be used to replace arguments that use
Szemer\'edi's regularity lemma. This is not too surprising, as in both
cases the strength of the result comes from the fact that the bound on
the quasirandomness can be made so small that it is even small
compared with the ``natural scale" of the structured part. Similarly,
it can be used to replace a version of Szemer\'edi's
regularity lemma, due to Green \cite{benarithreg}, that concerns dense
subsets of finite Abelian groups.

\end{document}